\documentclass[12pt,reqno]{amsart}
\usepackage{graphicx} 
\usepackage{amsmath}
\usepackage{amssymb}
\usepackage[margin=1in]{geometry}
\usepackage{mathtools}

\usepackage[dvipsnames]{xcolor}
\usepackage{hyperref}
\hypersetup{
    colorlinks=true,
    linkcolor=cyan!80!black,
    citecolor=MidnightBlue,
    urlcolor=magenta,
}
\usepackage{url}
\usepackage{booktabs}
\usepackage{dsfont}
\usepackage{natbib}

\theoremstyle{plain}
\newtheorem{theorem}{Theorem}[section]
\newtheorem{lemma}{Lemma}[section]
\theoremstyle{definition}
\newtheorem{example}{Example}[section]

\newtheorem{assumption}{Assumption}[section]
\theoremstyle{remark}
\newtheorem{remark}{Remark}[section]

\newcommand{\bbI}{\mathbb{I}}
\newcommand\norm[1]{\lVert #1 \rVert}
\newcommand{\MP}{\mathsf{MP}}
\newcommand{\SC}{\mathsf{SC}}
\renewcommand{\P}{\mathbb{P}}
\newcommand{\bbE}{\mathbb{E}}
\newcommand{\bbC}{\mathbb{C}}
\newcommand{\cL}{\mathcal{L}}
\newcommand{\op}{\mathrm{op}}
\newcommand{\hs}{\mathrm{HS}}
\newcommand{\Leb}{\mathsf{Leb}}
\newcommand{\bone}{\mathbf{1}}
\newcommand{\convas}{\xrightarrow{\text{a.s.}}}

\newcommand{\convd}{\xrightarrow{d}}

\DeclareMathOperator{\diag}{diag}
\DeclareMathOperator{\V}{Var}
\DeclareMathOperator{\Tr}{Tr}

\makeatletter
\let\save@mathaccent\mathaccent
\newcommand*\if@single[3]{%
  \setbox0\hbox{${\mathaccent"0362{#1}}^H$}%
  \setbox2\hbox{${\mathaccent"0362{\kern0pt#1}}^H$}%
  \ifdim\ht0=\ht2 #3\else #2\fi
  }
\newcommand*\rel@kern[1]{\kern#1\dimexpr\macc@kerna}
\newcommand*\widebar[1]{\@ifnextchar^{{\wide@bar{#1}{0}}}{\wide@bar{#1}{1}}}
\newcommand*\wide@bar[2]{\if@single{#1}{\wide@bar@{#1}{#2}{1}}{\wide@bar@{#1}{#2}{2}}}
\newcommand*\wide@bar@[3]{%
  \begingroup
  \def\mathaccent##1##2{%
    \let\mathaccent\save@mathaccent
    \if#32 \let\macc@nucleus\first@char \fi
    \setbox\z@\hbox{$\macc@style{\macc@nucleus}_{}$}%
    \setbox\tw@\hbox{$\macc@style{\macc@nucleus}{}_{}$}%
    \dimen@\wd\tw@
    \advance\dimen@-\wd\z@
    \divide\dimen@ 3
    \@tempdima\wd\tw@
    \advance\@tempdima-\scriptspace
    \divide\@tempdima 10
    \advance\dimen@-\@tempdima
    \ifdim\dimen@>\z@ \dimen@0pt\fi
    \rel@kern{0.6}\kern-\dimen@
    \if#31
      \overline{\rel@kern{-0.6}\kern\dimen@\macc@nucleus\rel@kern{0.4}\kern\dimen@}%
      \advance\dimen@0.4\dimexpr\macc@kerna
      \let\final@kern#2%
      \ifdim\dimen@<\z@ \let\final@kern1\fi
      \if\final@kern1 \kern-\dimen@\fi
    \else
      \overline{\rel@kern{-0.6}\kern\dimen@#1}%
    \fi
  }%
  \macc@depth\@ne
  \let\math@bgroup\@empty \let\math@egroup\macc@set@skewchar
  \mathsurround\z@ \frozen@everymath{\mathgroup\macc@group\relax}%
  \macc@set@skewchar\relax
  \let\mathaccentV\macc@nested@a
  \if#31
    \macc@nested@a\relax111{#1}%
  \else
    \def\gobble@till@marker##1\endmarker{}%
    \futurelet\first@char\gobble@till@marker#1\endmarker
    \ifcat\noexpand\first@char A\else
      \def\first@char{}%
    \fi
    \macc@nested@a\relax111{\first@char}%
  \fi
  \endgroup
}
\makeatother

\let\hat\widehat
\let\tilde\widetilde
\let\bar\widebar

\title{Bulk Spectra of Truncated Sample Covariance Matrices}
\author[S. Ghosh]{Subhroshekhar Ghosh}
\address{
    Department of Mathematics \\
    National University of Singapore \\
    10 Lower Kent Ridge Road \\
    Singapore 119076
}
\email{matghos@nus.edu.sg}

\author[S. S. Mukherjee]{Soumendu Sundar Mukherjee}
\address{
    Statistics and Mathematics Unit \\
    Indian Statistical Institute \\
    203 B.T. Road, Kolkata 700108 \\
    West Bengal, India
}
\email{ssmukherjee@isical.ac.in}

\author[H. Talukdar]{Himasish Talukdar}
\address{
    Statistics and Mathematics Unit \\
    Indian Statistical Institute \\
    203 B.T. Road, Kolkata 700108 \\
    West Bengal, India
}
\email{talukdar.himasish@gmail.com}

\begin{document}

\begin{abstract}
Determinantal Point Processes (DPPs), which originate from quantum and statistical physics, are known for modelling diversity. Recent research \citep{ghosh2020gaussian} has demonstrated that certain matrix-valued $U$-statistics (that are truncated versions of the usual sample covariance matrix) can effectively estimate parameters in the context of Gaussian DPPs and enhance dimension reduction techniques, outperforming standard methods like PCA in clustering applications. This paper explores the spectral properties of these matrix-valued $U$-statistics in the \emph{null} setting of an isotropic design. These matrices may be represented as $X L X^\top$, where $X$ is a data matrix and $L$ is the Laplacian matrix of a random geometric graph associated to $X$. The main mathematically interesting twist here is that the matrix $L$ is dependent on $X$. We give complete descriptions of the bulk spectra of these matrix-valued $U$-statistics in terms of the Stieltjes transforms of their empirical spectral measures. The results and the techniques are in fact able to address a broader class of kernelised random matrices, connecting their limiting spectra to generalised Mar\v{c}enko-Pastur laws and free probability.
\end{abstract}

\maketitle

\section{Introduction}\label{sec:intro}
The explosion of large-scale data, often referred to as ``big data'', has transformed industries, research fields, and everyday life in the recent years. The phenomenon of massive scale data has called for new approaches to modelling and analysis. In particular, the question of diverse samples to enable a more parsimonious representation of data has led to connections with statistical physics, wherein models of strongly repulsive particle systems have been leveraged to augment the diverseness of features in machine learning procedures.

A key model in that respect is that of \textit{determinantal point processes} or DPPs. A DPP is a probability distribution over subsets of a given ground set, such that the probability of a subset is proportional to the determinant of a kernel matrix corresponding to the subset. DPPs are known for their ability to model diversity, making them useful for selecting a set of items that are \textit{spread out} over the feature space. Originating in quantum and statistical physics, DPPs have quickly grown to have an increasing impact as a significant component of a machine learning toolbox based on negative dependence.

A major parametric model of DPPs that has attracted attention in recent years is that of the \textit{Gaussian Determinantal Processes}, abbrv. GDP \citep{ghosh2020gaussian}. In particular, it was shown in \cite{ghosh2020gaussian} that a certain matrix-valued statistic
\begin{equation}\label{eq:est_GDP}
    \hat{\Sigma} = \frac{1}{2n^2}\sum_{1 \le i, j \le n} \bbI(\|X_i - X_j\| \le r) \, (X_i - X_j)(X_i - X_j)^\top,
\end{equation}
where $(X_i)_{i=1}^n \subset \mathbb{R}^p$ are data points and $r$ is a suitably chosen threshold, effectively performs parameter estimation in the GDP model.
$\hat{\Sigma}$ may be viewed as a certain truncation of the sample covariance matrix (based on pairwise distances between data points), Further, it was demonstrated empirically in \cite{ghosh2020gaussian} that this matrix-valued test statistic can be leveraged as an ansatz to build dimension reduction tools that arguably  outperform standard PCA based methods, especially in the context of clustering applications.

\begin{figure}[t!]
\centering
\begin{tabular}{cc}
    \includegraphics[scale = 0.25]{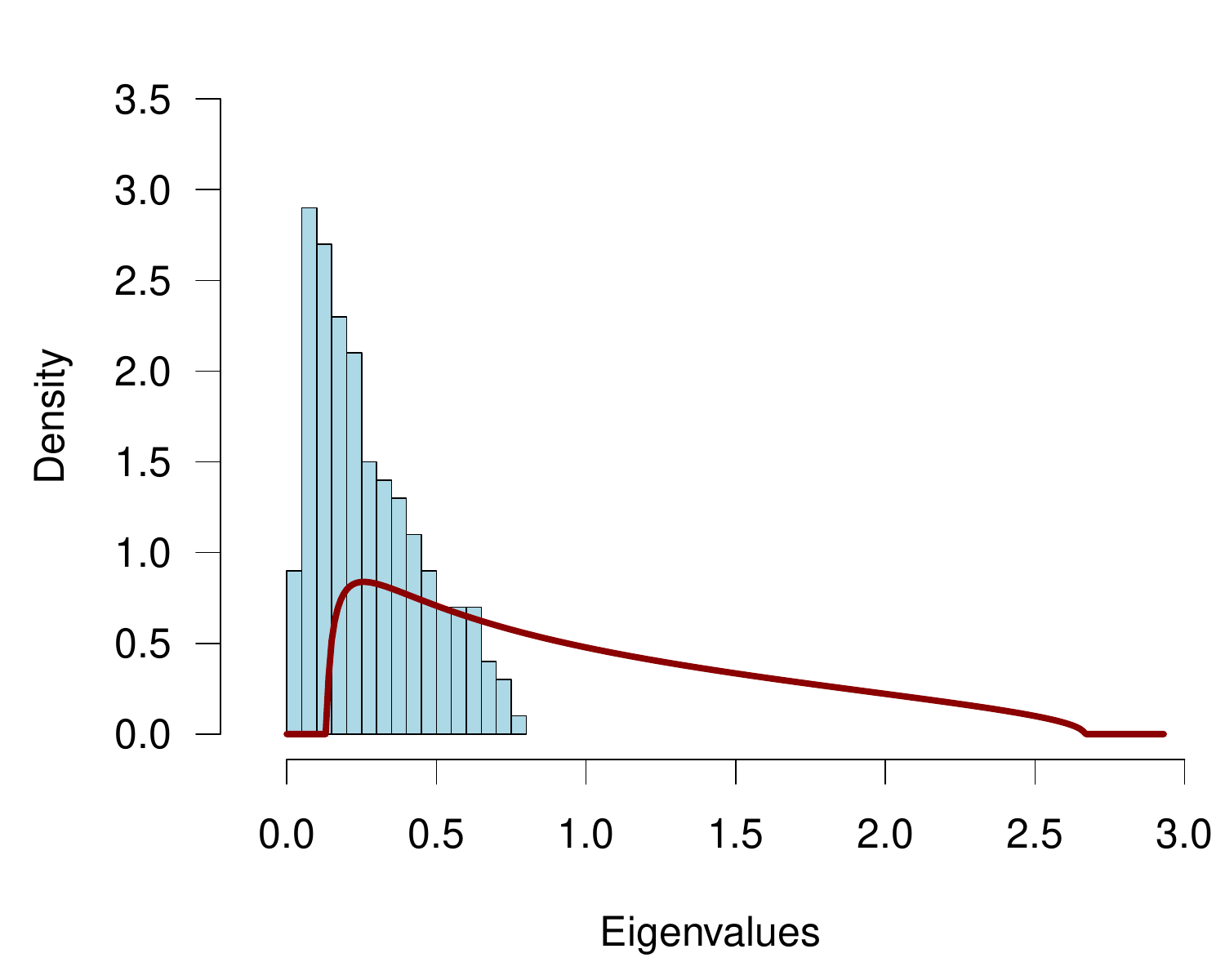} & \includegraphics[scale = 0.25]{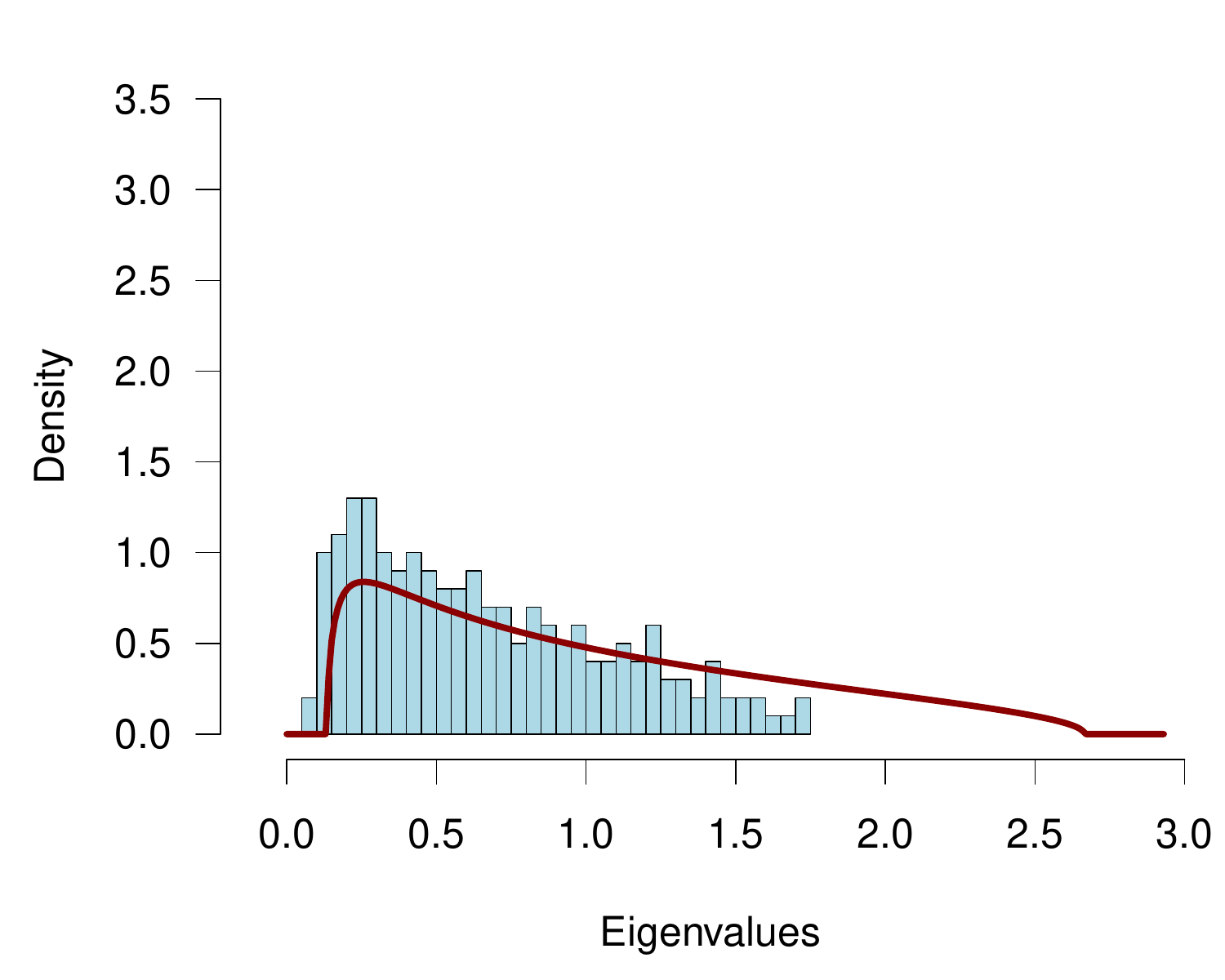} \\
    (a) $\beta = -0.1$  & (b) $\beta = 0.1$ \\
    \includegraphics[scale = 0.25]{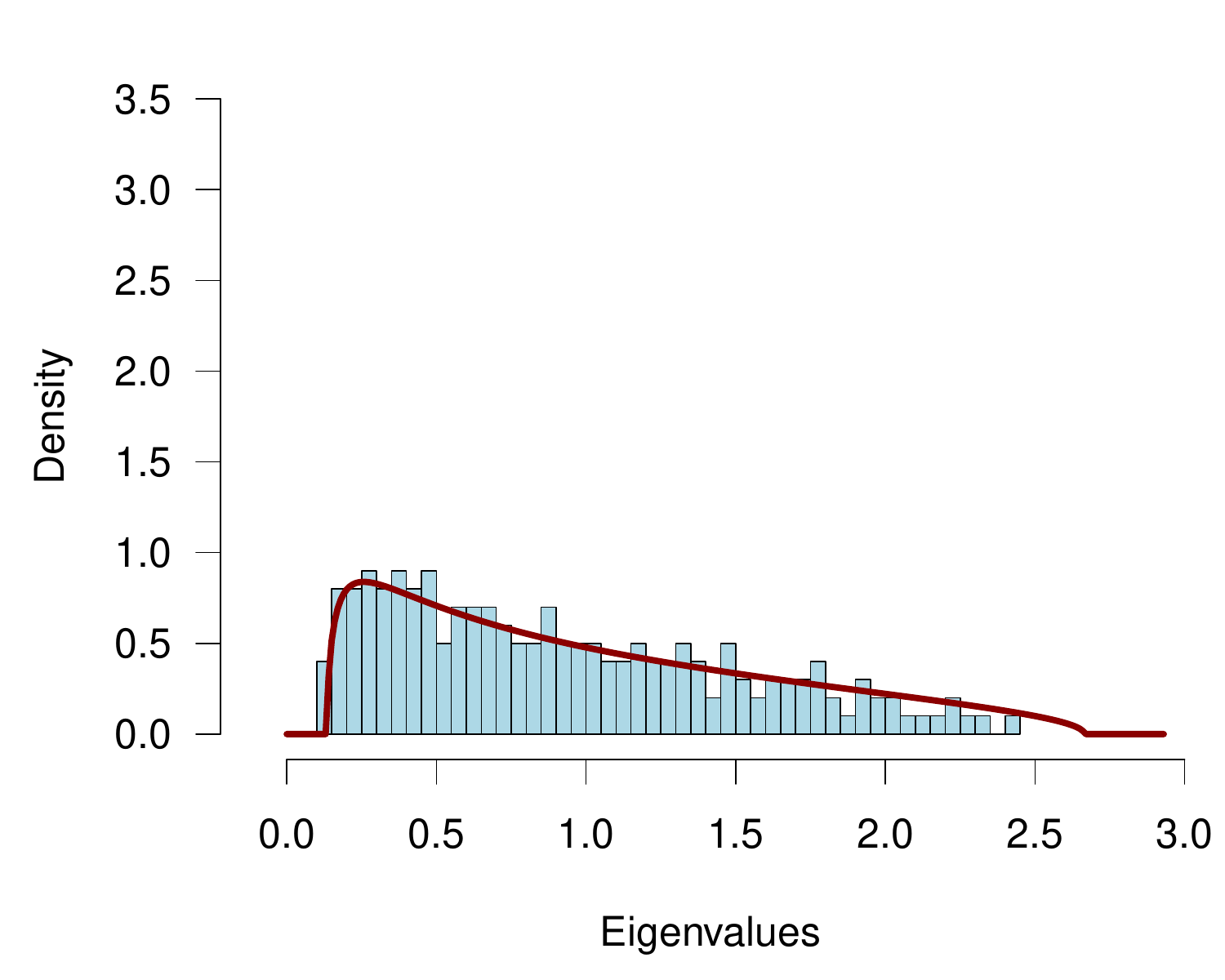} & \includegraphics[scale = 0.25]{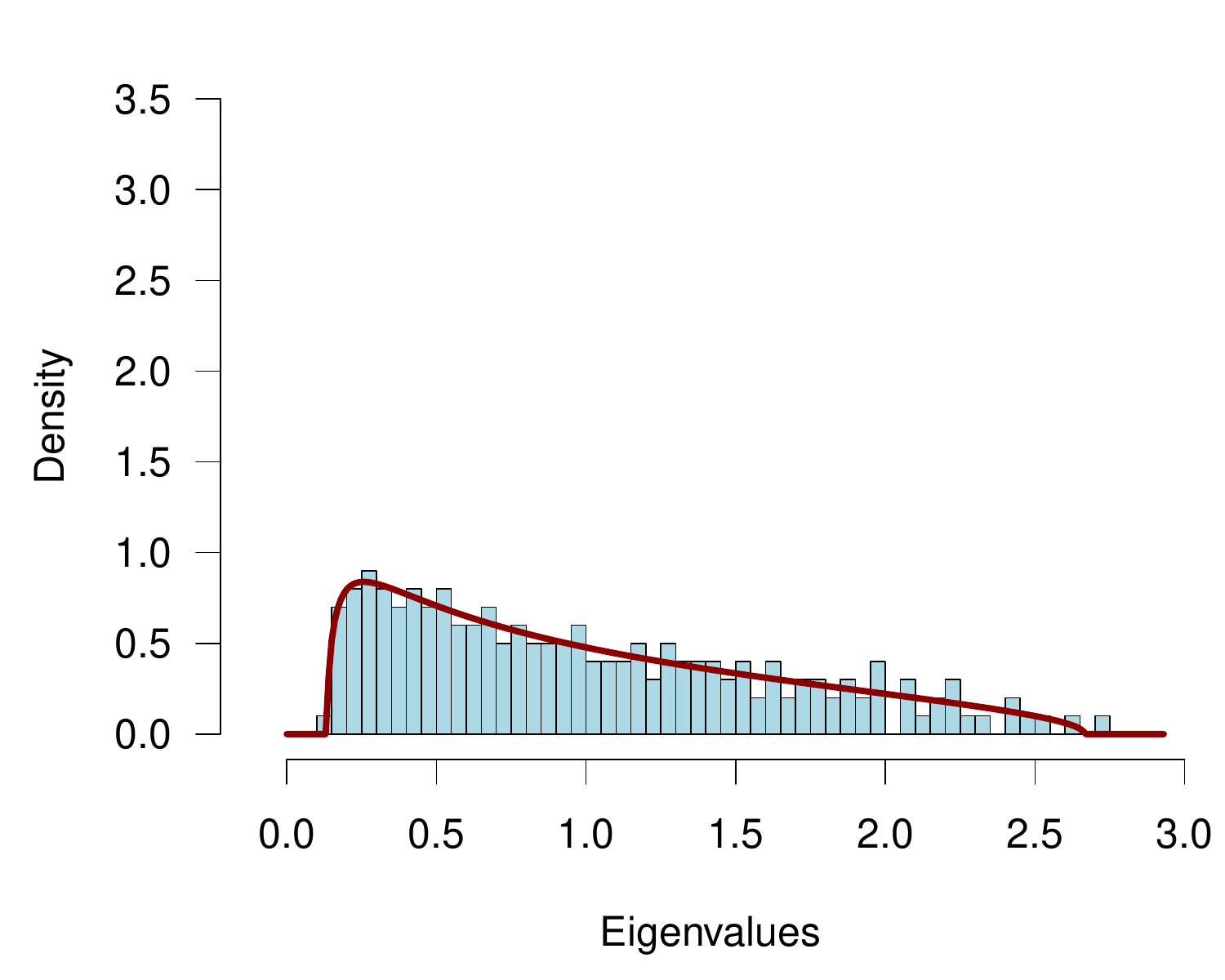} \\
    (c) $\beta = 0.3$ & (d) $\beta = \infty$
\end{tabular}
\caption{Empirical spectral distribution of $\hat{\Sigma}$ for $r = r(\beta) = \sqrt{(2 + \beta) \sigma^2 p}$ (one realisation for each value of $\beta$). Here $p  = 200$, $n = 500, \sigma = 1$. The red curves denote the density of $\MP_{p/n, \sigma^2}$. The case $\beta = \infty$ corresponds to the sample covariance matrix.}
\label{fig:simu_1}
\end{figure}

An understanding of this latter phenomenon would require an understanding of the spectrum of the matrix in \eqref{eq:est_GDP}. In this paper, we take a first step towards this by studying its bulk spectrum, modelling the data as i.i.d. centered random variables. This is the most fundamental and basic setting in which one first needs to understand the behaviour and properties of the matrix in \eqref{eq:est_GDP}.

In fact, we are able to analyse a broader class of matrix-valued statistics which widely generalises $\hat{\Sigma}$, by incorporating a general class of kernel functions $K(X_i,X_j)$ in lieu of the distance based cutoff function $\bbI(\|X_i - X_j\| \le r)$. We perform a detailed analysis of the bulk spectrum of this broad class of kernelised random matrices and obtain a concrete description of their limiting spectral distributions as the size of the dataset $n$ and the dimension $p$ go to $\infty$ in a way such that $p/n \to c \in (0, \infty)$ (the so-called \emph{proportional asymptotics regime}). In particular, in the \emph{smooth} case, where the kernelised interaction is a suitably regular function of their mutual interaction, we can explicitly characterise the limiting spectral distribution as a parameterised family of Mar\v{c}enko-Pastur laws (see Theorem~\ref{thm:conv_smooth}). In the non-smooth case, we obtain a certain generalised Mar\v{c}enko-Pastur law as the limiting spectral distribution (see Theorem~\ref{thm:conv_nonsmooth}). See Figure~\ref{fig:simu_1} for histograms of the bulk spectra of $\hat{\Sigma}$ for different choices of the threshold $r$. See Figure~\ref{fig:simu_2} for an example of a kernelised version of $\hat{\Sigma}$ with a smooth kernel. We also obtain the limiting spectral distribution in the semi-high-dimensional regime where $p / n \to 0$ and $p \gg \sqrt{n}$ (see Theorem~\ref{thm:conv_semi_high_dim}).

\begin{figure}[t!]
\centering
\begin{tabular}{cc}
    \includegraphics[scale = 0.25]{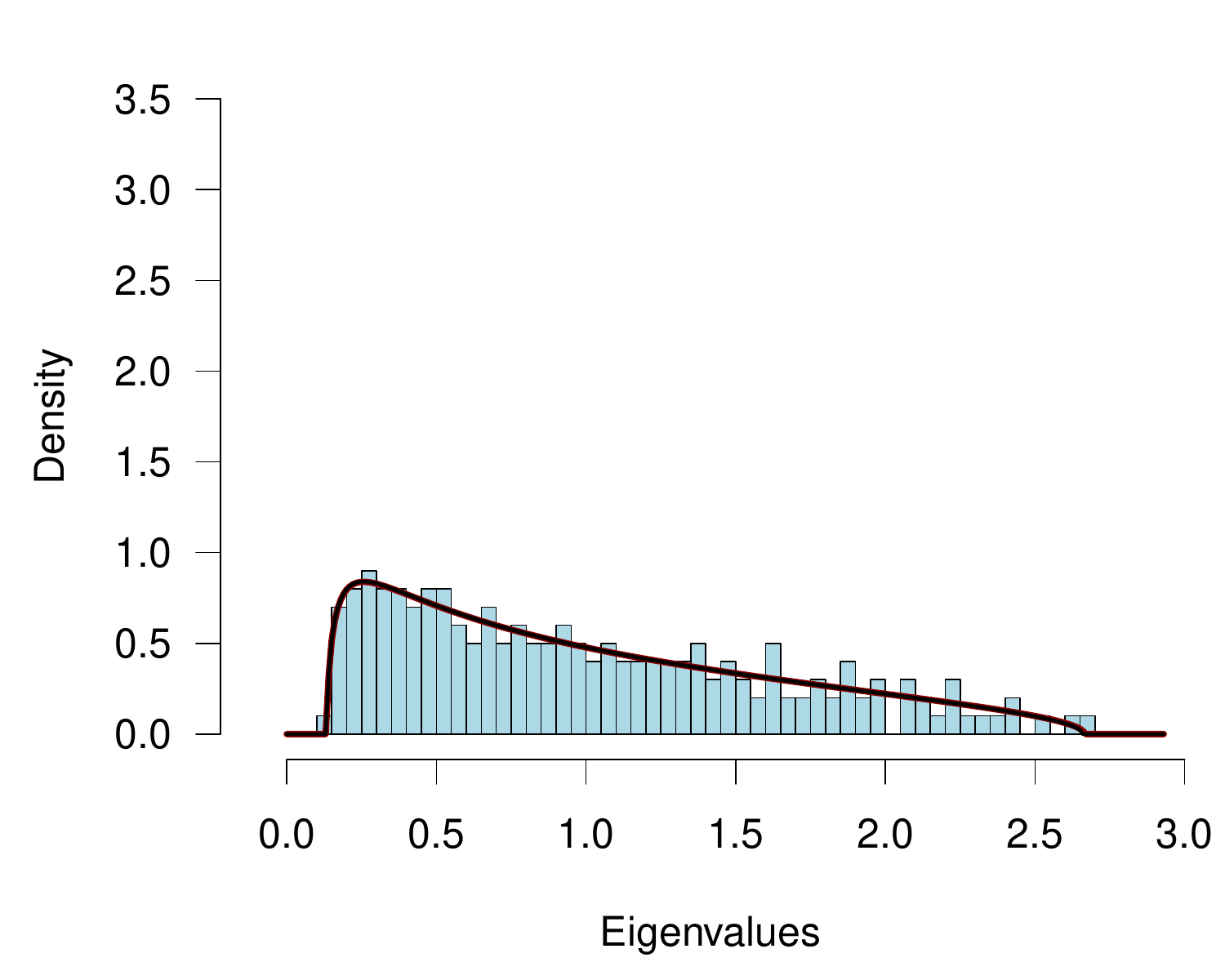} & \includegraphics[scale = 0.25]{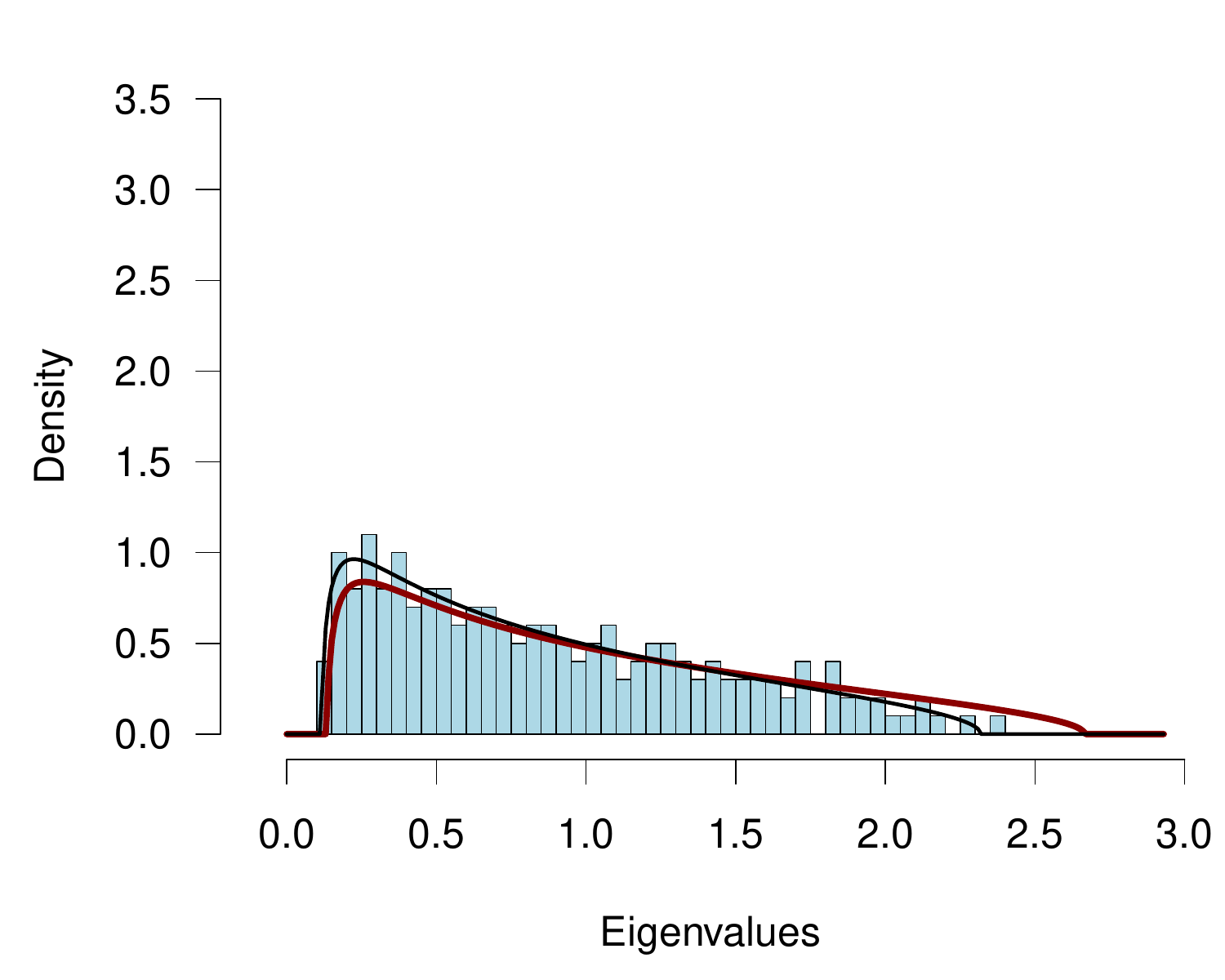} \\
    (a) $\tau = 0.4$  & (b) $\tau = 0.7$ \\
    \includegraphics[scale = 0.25]{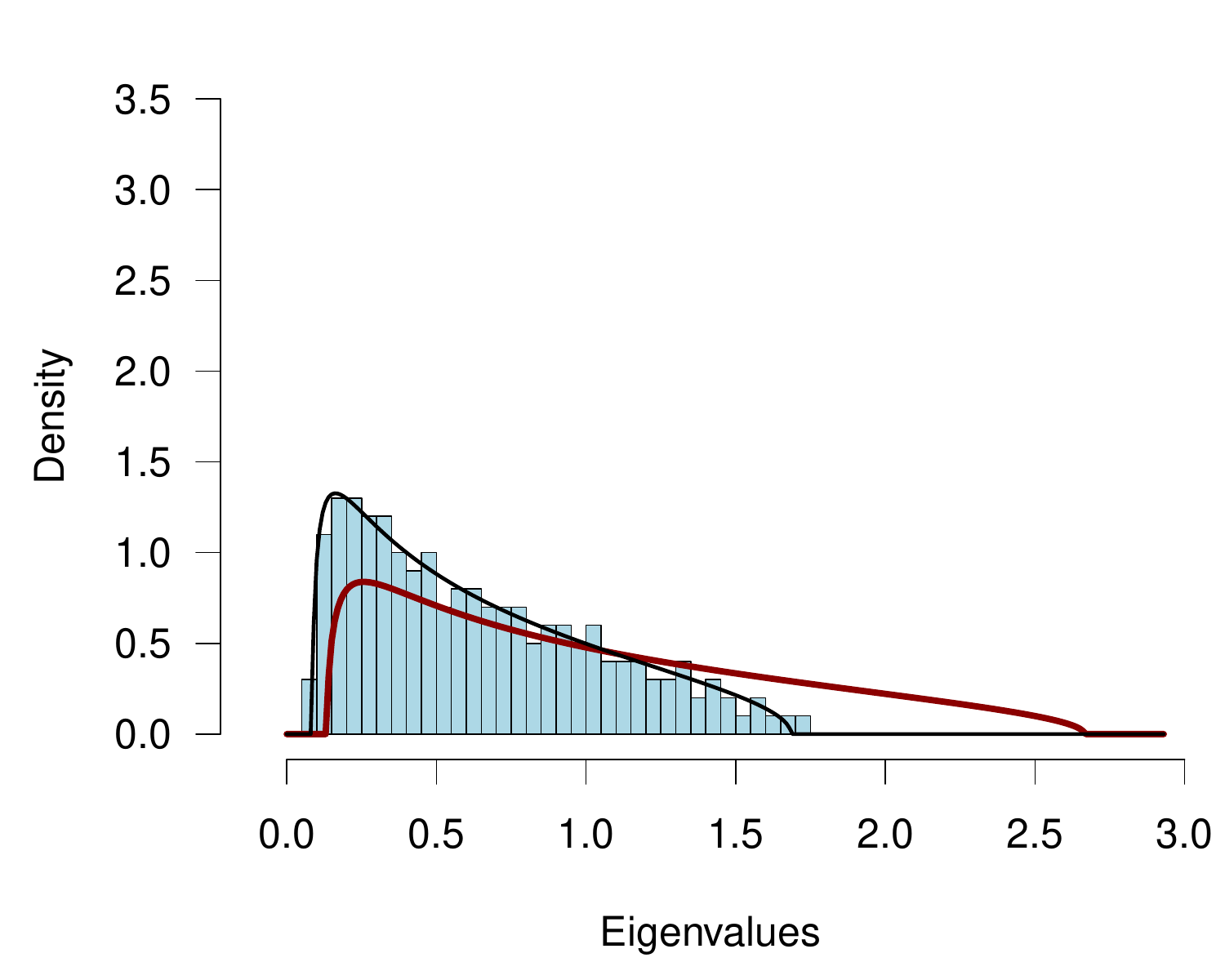} & \includegraphics[scale = 0.25]{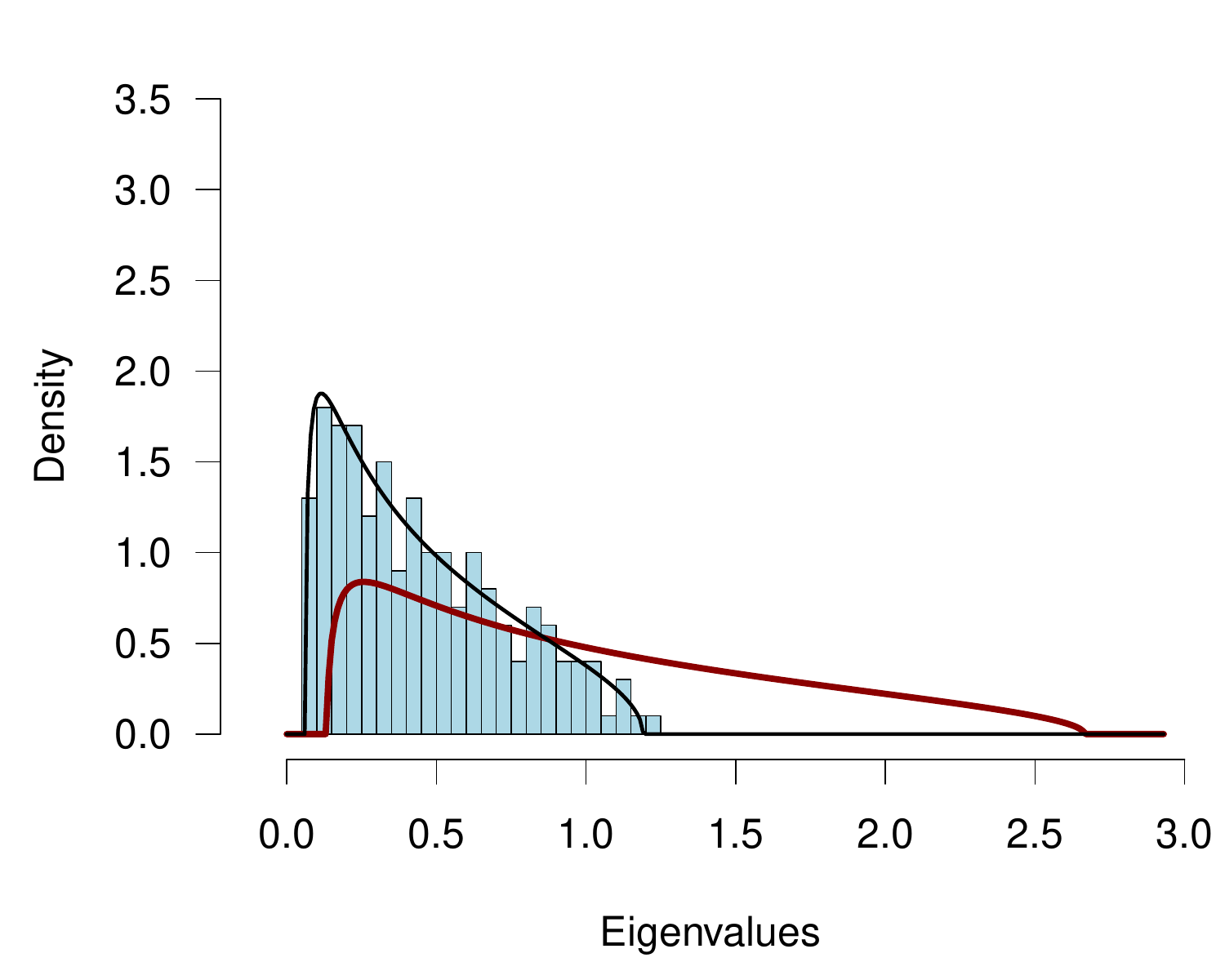} \\
    (c) $\tau = 1$ & (d) $\tau = 1.3$
\end{tabular}
\caption{Empirical spectral distribution of $M$ (defined in \eqref{eq:def_trunc_cov}) with the Gaussian kernel $K(x, y) = 1 - \exp\big(-\frac{\|x - y\|^2}{2 p \tau^2}\big)$ (one realisation for each value of $\tau$). Here $p  = 200$, $n = 500, \sigma = 1$. The red curves denote the density of $\MP_{p/n, \sigma^2}$. The black curves depict the density of $\MP_{p/n, (1 - \exp(-\sigma^2/\tau^2))^2 \sigma^2)}$. As $\tau$ approaches $0$, the spectrum approaches that of the sample covariance matrix.}
\label{fig:simu_2}
\end{figure}

The main insight that goes into analysing the spectrum of the matrix in \eqref{eq:est_GDP} is to represent it as a matrix-valued Rayleigh quotient:
\begin{equation}\label{eq:matrix_Rayleigh_quotient}
    \frac{1}{n^2} X L X^T,
\end{equation}
where $L$ is the Laplacian matrix of the \emph{random geometric graph} on $n$ vertices whose edges are given by $\bbI(\|X_i - X_j\| \le r)$. It is to be noted that matrices of the form $X A X^\top$, where $A$ is a positive semi-definite matrix \emph{independent of X} have been studied in the literature in great detail. For example, in \cite{bai1995genmarcenko}, the authors considered the matrix $\frac{1}{n}XAX^\top$, where $X$ is a $n \times p$ (note that the roles of $n$ and $p$ are reversed in their notation but this is only a cosmetic difference) matrix of i.i.d. entries with zero mean and unit variance and $A$ is a diagonal matrix having some deterministic limiting spectral distribution $\mu_A$. If $X$ and $A$ are independent and $\frac{p}{n} \to y$, then their result says that the above-mentioned matrix has a deterministic limiting spectral distribution, whose Stieltjes transform $s$ is described as the unique solution, in the upper half plane $\mathbb{C}^+ = \{z \in \mathbb{C} : \Im z > 0\}$, of the equation
\begin{equation}\label{eq:gen_marcenko}
    s(z) = \frac{1}{-z + y\int \frac{t d\, \mu_A(t)}{1+ts(z)}},
\end{equation}
for $z \in \mathbb{C}^+$. The resulting limiting spectral distribution is called a \emph{generalised Mar\v{c}enko-Pastur Law} which also admits the following free probabilistic interpretation: it is the free multiplicative convolution of $\mu_A$ and the Mar\v{c}enko-Pastur law. In a more recent work, under certain additional assumptions, \cite{knowles2017anisotropic} obtained local laws for the matrices $\frac{1}{n}X A X^\top$ and $\frac{1}{n} A^{1/2}XX^\top A^{1/2}$, where $A$ is a deterministic matrix. The crucial difference of our model from these existing works is that the Laplacian matrix $L$ is \emph{dependent} on $X$. As such we need use careful decoupling arguments to analyse its spectrum. Obtaining local laws under our setting is an interesting direction for future research.

It may also be observed that $\hat{\Sigma}$, and its kernelised generalizations, belong to the wider family of \textit{matrix-valued $U$-statistics}. As such, our results also contribute to the burgeoning theory of matrix-valued $U$-statistics and their applications. For instance, the spectrum of a matrix-valued Kendall's $\tau$ statistic was studied recently by \cite{bandeira2017marchenko}. They showed that if $X_1, X_2, \ldots, X_n$ are i.i.d. $p$-dimensional random vectors with independent entries from a continuous distribution, then the empirical spectral distribution of the matrix-valued Kendall's $\tau$ statistic, defined as
\[
    \tau = \frac{1}{\binom{n}{2}} \sum_{1\leq i <j \leq n} \mathrm{sign}(X_i - X_j) \, \mathrm{sign}(X_i - X_j)^\top,
\]
converges weakly to $\frac{1}{3} + \frac{2}{3}Y$, in probability, where $Y$ follows the standard Mar\v{c}enko-Pastur distribution (here the sign function is applied componentwise). The proof heavily relies on a matrix version of the Hoeffding decomposition for $U$-statistics. Although, the matrix \eqref{eq:est_GDP} is also a matrix-valued $U$-statistic, the presence of the cutoff factor makes a direct use of Hoeffding decomposition difficult. Instead, we directly analyse the Stieltjes transform of the empirical spectral distribution.

The rest of the paper is organised as follows. In Section~\ref{sec:prelim}, we describe the model under consideration and recall preliminaries of random matrices. In Section~\ref{sec:main} we state our main results and work out some examples. We also provide brief proof sketches of our main results in this section. Section~\ref{sec:proofs} gives detailed proofs of all the results. Finally, in Appendix~\ref{sec:aux}, we collect some useful results from matrix analysis and concentration of measure which are used throughout the paper.

\section{The model}\label{sec:prelim}
Suppose $w_{ij}$, $i \in [p], j \in [n]$ are i.i.d. random variables on some probability space $(\Omega, \mathcal{F}, \mathbb{P})$. Assume $\bbE w_{11} = 0$,  $\V(w_{11}) = \sigma^2$ and $\bbE w^4_{11} < \infty$. Define the $p$-dimensional vectors $X_j= (w_{1j}, w_{2j},\ldots , w_{pj})^\top$, $j = 1, 2, \ldots, n$. $X$ is the $p \times n$ matrix with $X_j$'s as columns. Also define $\bar{X} = \frac{1}{n} \sum_{i = 1}^n X_i$.

We consider two asymptotic regimes:
\begin{enumerate}
    \item The proportional asymptotics regime: $\frac{p}{n} \to c \in (0, \infty)$.
    \item The semi-high-dimensional regime: $\frac{p}{n} \to 0$.
\end{enumerate}

Suppose that $K_p:\mathbb{R}^p \times \mathbb{R}^p \to [0,1]$ is a function symmetric in its coordinates, that is $K_p (U, V) = K_p (V, U)$. Let $A$ denote the $n \times n$ symmetric random matrix with entries $A_{ij} = K_p (X_i, X_j)$.

The Empirical Spectral Distribution (ESD) of a real symmetric matrix $Y_{n \times n}$ is defined as
\[
    \mu_Y = \frac{1}{n}\sum_{i = 1}^{n}{\delta_{\lambda_i}},
\]
where $\lambda_1, \lambda_2, \ldots, \lambda_n$ are the eigenvalues of $Y$. The weak limit of the ESD (defined almost surely or in probability depending on the context) is called the Limiting Spectral Distribution (LSD).

In this paper we are interested in the \emph{truncated covariance matrix}
\begin{equation}\label{eq:def_trunc_cov}
    M = \frac{1}{2n^2} \sum_{1 \leq i, j \leq n} A_{ij} (X_i - X_j)(X_i - X_j)^\top,
\end{equation}
which is a generalisation of the estimator in \eqref{eq:est_GDP}.

Notice that if $K \equiv 1$, then
\[
    M = \frac{1}{2n^2} \sum_{1 \leq i, j \leq n} (X_i - X_j)(X_i - X_j)^\top = \frac{1}{n} \sum_{i = 1}^n (X_i - \bar{X})(X_i - \bar{X})^\top,
\]
which is the sample-covariance matrix of the observations $X_1, \ldots, X_n$. As it is a rank-$1$ perturbation of the matrix $\frac{1}{n} X X^\top$ (which will also be called the sample-covariance matrix), they share the same LSD.

In the proportional asymptotic regime, it is well known that the sample-covariance matrix $\frac{1}{n} X X^\top$ has as its LSD the Mar\v{c}henko-Pastur distribution $\MP_{c, \sigma^2}$ with parameters $(c, \sigma^2)$. Recall that when $c \in (0, 1]$, $\MP_{c, \sigma^2}$ has density
\[
    d\MP_{c, \sigma^2}(x) = \frac{1}{2 \pi \sigma^2} \frac{\sqrt{(b - x)(x - a)}}{cx} \mathbb{I}_{(a, b)}dx,
\]
where $a = \sigma^2 (1 - \sqrt{c})^2$ and $b = \sigma^2(1 + \sqrt{c})^2$. When $c > 1$, $\MP_{c, \sigma^2}$ has a mass of $(1 - \frac{1}{c})$ at $0$, the remaining part has the same density as above, i.e.
\[
    \MP_{c, \sigma^2} = \bigg(1 - \frac{1}{c}\bigg) \delta_0 + \frac{1}{c} \nu,
\]
where
\[
    d\nu(x) = \frac{1}{2 \pi \sigma^2} \frac{\sqrt{(b - x)(x - a)}}{cx} \mathbb{I}_{(a, b)}dx.
\]
We will see that when the kernel $K_p$ is Lipschitz and the entries $w_{ij}$ satisfy some regularity conditions, then the LSD of $M$ is a scaled Mar\v{c}enko-Pastur law (see Theorem~\ref{thm:conv_smooth}). However, if $K_p$ is non-smooth, then a different LSD emerges, which is a generalised Mar\v{c}enko-Pastur law (see Theorem~\ref{thm:conv_nonsmooth}).

On the other hand, in the semi-high-dimensional regime, one requires different scaling and centering. In is well known that the ESD of $\sqrt{\frac{n}{p}}(\frac{1}{n}XX^\top - I)$ converges to the standard semi-circle law (see \cite{bai1988semicircle}, p. 864). The semi-circle law $\SC_{\varpi^2}$ with variance $\varpi^2 > 0$ is defined as
\[
    d\SC_{\varpi^2} (x) = \frac{1}{2 \pi \varpi^2} \sqrt{4\varpi^2 - x^2} \,\, \bbI(|x| \le 2\varpi) \, dx.
\]
For $\varpi = 1$, we have the standard semi-circle law.

In our setup, we prove that under certain conditions on the moments of $A_{12}$, in the regime
\begin{equation}\label{ass:dimension}
    p \gg \sqrt{n},
\end{equation}
the ESD of
\begin{equation}\label{eq:def_normalised_trunc_cov}
    E = \sqrt{\frac{n}{p}}(M - \alpha_p\sigma^2 I)
\end{equation}
converges weakly to a semi-circle law with parameters depending on $\bbE A_{12}^2$ and $\sigma^2$, almost surely. We also describe the Stieltjes transform of the limiting distribution (see Theorem~\ref{thm:conv_semi_high_dim}).

\section{Main results}\label{sec:main}
\subsection{The non-smooth case}
Let $d : \mathbb{R}^2 \to \mathbb{R}$ be a symmetric function such that
\begin{equation}
    \mathbb{E}|d(w_{11},w_{12})|^3< \infty.
\end{equation}
Typical examples of $d(x,y)$ are $(x-y)^2$ or $|x-y|$. We define $d_p:\mathbb{R}^p\times \mathbb{R}^p\to \mathbb{R}$, $d_p(x,y)=\sum_{i=1}^{p} d(x_i,y_i)$. Let $\phi_p:\mathbb{R}\to [0,1]$ be monotonic and potentially dependent on $p$. For the purpose of the first theorem we shall assume that $K_p$ has the following form
\[
    K_p(x,y)=\phi_p(d_p(x,y)).
\]
Notice that this class of kernels includes the indicator kernel $\bbI(\norm {x-y}\leq r_p)$ and the Gaussian kernel $1- \exp\big(-\frac{\norm{x-y}^2}{2\tau^2_p}\big)$, where $r_p, \tau_p$ are suitable constants.

We shall also require some limiting properties of the sequence of functions $(\phi_p)$. We state them now. First fix the following notations:
\begin{align*}
    m_1 &=\bbE [d(X_{11},X_{12})]; \\
    m_2 &=\text{Var}(d(X_{11},X_{12})); \\
    m_2^{(1)} &= \text{Var}(\mathbb{E}[d(X_{11},X_{12})|X_{11}]); \\
    m_2^{(2)} &= \mathbb{E}\text{Var}(d(X_{11},X_{12})|X_{11}).
\end{align*}
Define the functions $\psi_p,\tilde{\phi}_p: \mathbb{R}\to \mathbb{R}$ as $\psi_p(x)=pm_1+\sqrt{pm_2}x$ and $\tilde{\phi}_p=\phi_p\circ\psi_p$.
\begin{assumption}\label{ass:phi_condition}
Suppose there exists $\tilde{\phi}$ such that for any $\epsilon>0$,
\begin{equation}
    \Leb(|\tilde{\phi}_p-\tilde{\phi}|>\epsilon)\to 0
\end{equation}
as $p \to \infty$ where $\Leb$ is the Lebesgue measure. In other words, $\tilde{\phi}_p$ converges in Lebesgue measure to $\tilde{\phi}$.
\end{assumption}

Our first theorem describes the LSD of $M$ in terms of its Stieltjes transform. Throughout the paper $z$ will denote a complex number with $u=\Re z$ and $v=\Im z$, i.e. $z = u + \iota v$. Recall that the Stieltjes transform $S_{\mu}$ of a probability measure $\mu$ on $\mathbb{R}$ is a complex function defined for $z\in\mathbb{C}^+$ as follows:
\[
    S_\mu(z) := \int\frac{d\mu(x)}{x-z}.
\]
Suppose $\{\mu_n\}_{n \geq 1}, \mu$ are probability measures on $\mathbb{R}$ with Stieltjes transforms $\{S_{\mu_n}\}_{n \ge 1}$ and $S_\mu$, respectively. It is well known that $S_{\mu_n}\to S_\mu$ pointwise on $\mathbb{C}^+$ if and only if $\mu_n \to \mu$ weakly (see, e.g., \cite{anderson2010introduction}). Moreover, if $\{\mu_n\}_{n \geq 1}$ are random probability measures and $\mu$ is a deterministic probability measure, then $S_{\mu_n} (z) \to S_\mu (z)$ almost surely for each fixed $z \in \mathbb{C}^+$ if and only if $\mu_n \to \mu$ almost surely.
\begin{theorem}\label{thm:conv_nonsmooth}
Suppose that Assumption~\ref{ass:phi_condition} holds. Then the ESD of $M$ converges weakly to a deterministic distribution, almost surely. Moreover, if $s(z)$ is the Stieltjes transform of the limiting distribution, then $s(z)$ is the unique solution in $\mathbb{C}^+$ of the following equation:
\begin{equation}\label{eq:non_smooth_ST}
    1 + zs(z) = \mathbb{E}_{\zeta}\bigg[\frac{\sigma^2 s(z) \zeta }{1+c\sigma^2 s(z) \zeta }\bigg],
\end{equation}
where
\begin{equation}\label{eq:zeta_defn}
    \zeta = \bbE_{Z_2} \bigg[\tilde{\phi}\bigg(\sqrt{\frac{m_2^{(1)}}{m_2}}Z_1 + \sqrt{\frac{m_2^{(2)}}{m_2}}Z_2\bigg) \bigg],
\end{equation}
with $Z_1, Z_2$ being i.i.d. $N(0,1)$ random variables.
\end{theorem}
\begin{remark}
Suppose $\xi_1, \xi_2, \ldots, \xi_n$ are i.i.d. bounded random variables such that $\xi_i$ and $X_j$ are independent for all $i \neq j$. Further assume that $\xi_1$ converges in distribution to some variable $\zeta$. Then, the proof of Theorem~\ref{thm:conv_nonsmooth} will show that the LSD of $\frac{1}{n}\sum_{i=1}^n \xi_i X_i X_i^\top$ is given by \eqref{eq:non_smooth_ST}. This result is known if we further assume that $\xi_i$ is independent of $X_i$, but here we allow them to depend.
\end{remark}
\begin{remark}
If $M$ is represented as a matrix-valued Rayleigh quotient $\frac{1}{n^2} X L X^\top$, our proof will show that the ESD of $\frac{1}{n}L$ will converge weakly to $\zeta$ as defined in \eqref{eq:zeta_defn}. Now, a moment's thought will reveal that the equations \eqref{eq:gen_marcenko} and \eqref{eq:non_smooth_ST} are equivalent once we make the necessary adjustments for the scaling. In other words, even though $X$ and $L$ are dependent and $L$ is not diagonal, a generalised Mar\v{c}enko-Pastur law emerges as the LSD.
\end{remark}
\begin{example}[Indicator kernel]\label{ex:indicator_kernel}
We first consider kernel
\[
    K_p(x, y) = \bbI(\|x - y\| \leq r_p),
\]
where $r_p$ is an appropriate threshold. This gives us the estimator \eqref{eq:est_GDP} that motivated the present study. In this case, we will assume that $w_{11}$ is Gaussian. At the least, we shall need that $\alpha_{p} \coloneqq\bbE K_p(X_1, X_2)$ converges to a nonzero quantity as $p \to \infty$. The suitable choice for $r_p$ turns out to be
\[
    r_p^2= ((2p+2\sqrt{2p}z_\alpha) \sigma^2+o(\sqrt[2]{p})).
\]
For $\alpha\in(0,1)$, $z_\alpha=\Phi^{-1}(\alpha)$, where $\Phi$ is the distribution function of the standard normal variable. Observe that $\frac{\norm{X_1-X_2}^2}{2\sigma^2}\sim \chi^2_p$. Using the central limit theorem, $U=\frac{\frac{\norm{X_1-X_2}^2}{2\sigma^2}-p}{\sqrt{2p}}$ converges, in distribution, to a standard Gaussian. Notice that
\begin{align*}
    \alpha_{p} &= \mathbb{P}(\norm{X_1-X_2}\leq r_p) \\
    &=\mathbb{P}\bigg(\frac{\frac{\norm{X_1-X_2}^2}{2\sigma^2}-p}{\sqrt{2p}}\leq\frac{\frac{r_p^2}{2\sigma^2}-p}{\sqrt{2p}}\bigg) \\
    &= \mathbb{P}\left(U\leq z_\alpha+ o(1)\right) \\
    &= \alpha + o(1),
\end{align*}
as $p\to \infty$. Now, in this case, $\phi_p (t) = I(t \leq r_p^2)$ and $d(x, y) = (x-y)^2$. An easy computation shows that
\begin{align*}
    m_1 = 2 \sigma^2, & \quad m_2 = 8 \sigma^4, \\
    m_2^{(1)}= 2 \sigma^4, & \quad m_2^{(2)} = 6 \sigma^4.
\end{align*}
Then $\tilde{\phi}_p$ and $\tilde{\phi}$ turn out to be as follows:
\begin{align*}
\tilde{\phi}_p(t) &= I( t \leq z_\alpha +o(1)), \\
\tilde{\phi} (t) &= I(t \leq z_\alpha).
\end{align*}
Hence the distribution of $\zeta$ may be described as
\begin{align*}
\zeta &= \frac{1}{\sqrt{2 \pi}}\int_\mathbb{R}\tilde{\phi}\left(\frac{1}{2}Z+ \frac{\sqrt{3}}{2}t\right)e^{-\frac{t^2}{2}}dt \\
&=  \Phi\bigg(-\frac{1}{\sqrt{3}}Z +\frac{2}{\sqrt{3}} z_\alpha\bigg)\\
&\overset{d}{=}  \Phi\bigg(\frac{1}{\sqrt{3}}Z +\frac{2}{\sqrt{3}} z_\alpha\bigg),
\end{align*}
where $Z \sim N(0,1)$.
\end{example}
In the next example, however, Theorem~\ref{thm:conv_nonsmooth} can not be applied.
\begin{example}[Gaussian kernel]\label{ex:gaussian_kernel}
Let us now consider the Gaussian kernel
\[
    K_p(x,y) = 1- e^{-\frac{\norm{x-y}^2}{2p\tau^2}},
\]
where we have taken $\tau_p^2 =2 p \tau^2$. Here, $\phi_p (t) = 1 - e^{-\frac{t}{2p\tau^2}}$ and $d(x,y) = (x-y)^2$. Let us also assume that $w_{11}$ is Gaussian. One can calculate that
\[
\tilde{\phi}_p(t) = 1-e^{-\frac{\sigma^2}{\tau}} e^{-\frac{\sqrt{2}\sigma^2 t}{\sqrt{p}\tau^2}},
\]
which does not satisfy Assumption~\ref{ass:phi_condition}.
\end{example}

In order to examine the Gaussian kernel, one needs to use the smoothness properties of the Gaussian kernel which was missing in our analysis. This leads us to the next theorem where the smoothness of the kernel is crucially used.
\subsection{The smooth case}
We say that $K_p$ is Lipschitz with Lipschitz constant $\kappa_p$ (or $\kappa_p$-Lipschitz in short) if
\[
    |K_p(x_1, y_1) - K_p(x_2, y_2)| \leq \kappa_p(\norm{x_1-x_2}+\norm{y_1-y_2}).
\]
In particular, if for $p \in \mathbb{N}$, $\phi_p : \mathbb{R} \to [0,1]$ is a $\kappa_p$-Lipschitz function and one takes $K_p(x,y) = \phi_p(\norm{x-y})$, then one can check that $K_p$ is also an $\kappa_p$-Lipschitz kernel. Following \cite{amini2021concentration}, we will make the following assumption on $w_{11}$.
\begin{assumption}[LC class property]\label{ass:lc_class}
Let $\omega > 0$. We require $w_{11}$ to satisfy either of the following three conditions:
\begin{itemize}
    \item[(a)] $w_{11} \overset{d}{=} \varphi(Z)$ for some Lipschitz function $\varphi$ with $\norm{\varphi}_{\mathrm{Lip}} \leq \omega$, where $Z$ is a standard normal variable.
    \item[(b)] $w_{11}$ has density uniformly bounded below by $1/\omega$.
    \item[(c)] $X_1$ is strongly log-concave with curvature $\geq 1/\omega^2$.
\end{itemize}
\end{assumption}
Define $\alpha_p = \bbE K(X_1, X_2)$.
\begin{theorem}\label{thm:conv_smooth}
Suppose that $w_{11}$ satisfies Assumption~\ref{ass:lc_class} with parameter $\omega_p$. Let $K_p$ be an $\kappa_p$-Lipschitz kernel, such that $\kappa_p\omega_p =o(1/\sqrt{\log(n)})$ and $\alpha_p \to \alpha$. Then, the ESD of $M$ converges weakly to $\MP_{c, \alpha^2\sigma^2}$, almost surely.
\end{theorem}

\begin{example}
Let us revisit Example~\ref{ex:gaussian_kernel} in light of Theorem~\ref{thm:conv_smooth}. Take $w_{11} \sim N(0, \sigma^2)$. In this case, one may take $\kappa_p = \frac{\sqrt{2}}{e\sqrt{p}\tau}$, $\alpha_{ p} = 1- (1 + 2\sigma^2/p\tau^2)^{-\frac{p}{2}}$ and $\alpha = 1 - e^{-\sigma^2 / \tau^2}$. Thus, by Theorem~\ref{thm:conv_smooth}, $\mu_M \xrightarrow{d} \MP_{c, \sigma^2 (1 - e^{-\sigma^2/\tau^2})^2}$, almost surely.
\end{example}

\subsection{The semi-high-dimensional regime}
In Theorem~\ref{thm:conv_nonsmooth} if we take $c$ to be 0, then the Stieltjes transform of the LSD turns out to be $\frac{1}{\alpha \sigma^2 - z}$ which corresponds to the measure $\delta_{\alpha \sigma^2}$. This shows that $E$ is the right matrix to look at.
Define $\beta_p ^2 \coloneqq \bbE A_{12}^2$. Indeed, we show that
\begin{theorem}\label{thm:conv_semi_high_dim}
    If $\alpha_p \to \alpha$, $\beta_p^2 \to \beta^2$ and $\bbE w_{11}^8 < \infty$, in the regime $p \gg \sqrt{n}$, the ESD of $E$ converges weakly to $\SC_{\beta \sigma^2}$, almost surely.
\end{theorem}

\begin{remark}
    Note that without additional assumptions on $|\alpha_p - \alpha|$, we can not determine the convergence of
    \begin{equation}
        E' = \sqrt{\frac{n}{p}}(M - \alpha \sigma^2 I),
    \end{equation}
    hence the appearance of $\alpha_p$ in \eqref{eq:def_normalised_trunc_cov} instead of $\alpha$. For instance, if  $|\alpha_p - \alpha| = O(1/\sqrt{p})$, then we can replace $E$ by $E'$ in Theorem~\ref{thm:conv_semi_high_dim}.
\end{remark}
\begin{example}
In the set up of Example~\ref{ex:indicator_kernel}, $\beta_p^2=\alpha_p \to \alpha$. Thus, the LSD of $E$, in this case is $\SC_{ \sqrt \alpha \sigma^2}$. As for Example~\ref{ex:gaussian_kernel}, $\beta_p$ turns out to be $1 - 2(1-\frac{2 \sigma^2}{p \tau^2})^{-p/2} + (1-\frac{4\sigma^2}{p\tau^2})^{-p/2}$ and the parameter for $\SC$ changes accordingly.
\end{example}
In the remainder of this section, we briefly sketch the proofs of Theorems~\ref{thm:conv_nonsmooth},  \ref{thm:conv_smooth} and \ref{thm:conv_semi_high_dim}.

\subsection{Proof sketches}\label{sec:proof_sketches}
In this section, we give brief sketches of the proofs.
\subsubsection{The non-smooth case}\label{sec:proof_sketch_non_smooth}
We first observe that the matrix $M$ can be written in the form $\frac{1}{n^2}XLX^\top$ where $L$ is Laplacian matrix of a random geometric graph. Since the entries of $L$ are bounded, it turns out to be sufficient to consider only $\frac{1}{n^2} X D X^\top$, where $D = \diag(L)$. Note that $D_{ii} = \sum_{j\neq i} K_p(X_i, X_j)$. This has rather high dependence on $X_i$ and low dependence on the $X_j$'s, for $j \neq i$. Observe that conditional on $X_i$, $\{ K_p(X_i, X_j) - \bbE [K_p(X_i, X_j) | X_i]\}_{j\neq i}$ are bounded, centered i.i.d. random variables. Using Hoeffding's inequality, we show that the sum $\sum_{j\neq i} (K_p(X_i, X_j)- \bbE [K_p(X_i, X_j) | X_i])$ is negligible. It thus boils down to finding the LSD of the matrix
\begin{equation}\label{eq:M_bar}
    \bar{M} = \sum_{i=1}^n \xi_i X_i X_i^\top,
\end{equation}
where the $\xi_i$'s are i.i.d and independent of $X_j, j\neq i$. One can still not use the existing results in the literature since $\xi_i$ is not independent of $X_i$. However, $\xi_i$ involves a sum of certain i.i.d. variables. This enables us to apply a Berry-Esseen bound, eventually providing us with the distributional convergence of $\xi_i$. We now look at the Stieltjes transform $S_{\bar{M}}$ of  $\bar{M}$. Using standard martingale techniques, we deduce that $S_{\bar{M}}(z) - \bbE[S_{\bar{M}}(z)] \to 0$ almost surely for each fixed $z \in \bbC^+$. With this our goal becomes to get a recursive formula for $\bbE[S_{\bar{M}}]$. This is done by using the Sherman-Morrison formula and various perturbation inequalities for matrices coupled with the fact that $\xi_i$ converges in distribution as $p \to \infty$.

\subsubsection{The smooth case}\label{sec:proof_sketch_smooth}
As in the non-smooth case, we start by writing $M$ in the Rayleigh quotient form $\frac{1}{n^2} X L X^\top$. Our goal is to show that $M$ has the same LSD as $\frac{1}{n^2} X \mathcal{L} X^\top$ where $\mathcal{L} = \bbE L = \alpha_n (nI-J)$, where $J$ is the $n\times n$ matrix with all entries equal to 1. This is done by using matrix perturbation inequalities and Theorem~1 from \cite{amini2021concentration}. Since $J$ is of rank $1$, we have a further simplification: we may just consider the matrix $\frac{\alpha_n}{n}X X^\top$, which clearly has $\MP_{c, \alpha^2 \sigma^2}$ as its LSD.

\subsubsection{The $p/n \to 0$ case.}
As in the proof of Theorem~\ref{thm:conv_nonsmooth}, we can show that the LSD of $E$ is the same as that of
\[
    \bar E_1 = \sum_{i=1}^n \xi_i (X_i X_i^\top - \sigma^2 I).
\]
We analyse this matrix via its Stieltjes transform and various matrix perturbation inequalities.

\section{Proofs}\label{sec:proofs}
\subsection{Proof of Theorem ~\ref{thm:conv_nonsmooth}}
Since zeroing out the diagonal entries of $A$ does not change the Laplacian $L = D - A$, we may redefine $A$ (with a slight abuse of notation) as follows:
\begin{align*}
    A &= (((1 - \delta_{ij})A_{ij})), \\
    D &= \diag(A \bone), \\
    L &= D - A,
\end{align*}
where $\bone$ is the vector with each entry equal to $1$. Notice that $L$ is the Laplacian matrix corresponding to the weighted adjacency matrix $A$. The basic observation that will help us find the LSD of $M$ is the following matrix Rayleigh quotient representation:
\[
    M = \frac{1}{n^2} X L X^\top.
\]
The idea is to divide $M$ into three parts such that one part determines the LSD and the rest are negligible. In order to do so, we further decompose $D$ into two diagonal matrices. Define $\xi_i = \bbE [K_p(X_i,V) | X_i]$, where $V\sim X_1$ is independent of $X_1, X_2,\ldots X_n$. Notice that
\begin{align*}
    D_{ii} &= \sum_{j\neq i} K_p(X_i,X_j) \\
           &= (n-1)\xi_i+\sum_{j\neq i}(K_p(X_i,X_j)-\bbE [K_p(X_i,X_j)|X_i]).
\end{align*}
Define the following quantities:
\begin{align*}
    \xi       & = (\xi_1,\xi_2, \ldots, \xi_n)^\top,\\
    \xi'_{ij} & = K_p(X_i,X_j)-\bbE [K_p(X_i,X_j)|X_i],\\
    \xi'_i    & = \sum_{j\neq i}\xi'_{ij},\\
    \xi'      & = (\xi'_1, \xi'_2,\ldots,\xi'_n)^\top,\\
    D_1       & = (n-1)\text{diag}(\xi),\\
    D_2       & = \text{diag}(\xi').
\end{align*}
With these notations set, one has $D = D_1 + D_2$. Now decompose $M$ as
\[
    M=\frac{1}{n^2} X D_1 X^\top + \frac{1}{n^2} X D_2 X^\top - \frac{1}{n^2} X A X^\top.
\]
Define $\tilde{M} = \frac{1}{n^2} X D_1 X^\top$. We will show that $M$ has the same LSD as $\tilde{M}$. Towards that end, let $d_{W_2}$ denotes the $2$-Wasserstein distance between probability measures $\mu_1$ and $\mu_2$ possessing finite second moments:
\[
    d_{W_2}(\mu_1,\mu_2) := \inf \sqrt{\bbE (Z_1 - Z_2)^2},
\]
where the infimum is taken over all possible couplings of $(Z_1, Z_2)$ with marginals $Z_1 \sim \mu_1$ and $Z_2 \sim \mu_2$.
\begin{lemma}\label{lem:non_smooth_reduction_1}
    We have $d_{W_2}(\mu_{M}, \mu_{\tilde{M}}) \convas 0$.
\end{lemma}
\begin{proof}
Using the Hoffman-Wielandt inequality (see Lemma~\ref{lem:hoffman-wielandt}) and the facts that
\[
    \| G H\|_{\hs} \le \min\{\|G\|_{\op} \|H\|_{\hs}, \|G\|_{\hs} \|H\|_{\op}\},
\]
and
\[
    \|X^\top\|_{\op} = \|X\|_{\op} = \sqrt{\|X^\top X\|_{\op}},
\]
we have the following estimate
\begin{align*}
    d_{W_2}(\mu_M, \mu_{\tilde M}) &\leq \frac{1}{\sqrt{p}}\norm{M - \tilde{M}}_{\hs} \\
    &= \frac{1}{n^2 \sqrt p} \norm{X(D_2-A) X^\top}_{\hs} \\
    &\leq \frac{1}{n^2 \sqrt p} \norm{X D_2 X^\top}_{\hs} +\frac{1}{n^2 \sqrt p} \norm{X  A X^\top}_{\hs} \\
    &\leq \frac{1}{n^2} \norm{X D_2 X^\top}_{\op}+ \frac{1}{n^2 \sqrt p} \norm{X^\top X}_{\op}  \norm{A}_{\hs}\\
    &\leq \frac{1}{n } \norm{X^\top X}_{\op} \left ( \frac{1}{n}\norm{D_2}_{\op}+  \frac{1}{n \sqrt p}\norm{A}_{\hs}\right ).
\end{align*}
Since the fourth moment of the entries is finite, using Theorem~3.1 of \cite{yin1988limit},
\[
    \frac{1}{n} \norm{X^\top X}_{\op} \convas (1+\sqrt{c})^2\sigma^2.
\]
Let us now consider $\norm{A}_{\hs}$ and $\norm{D_2}_{\op}$. Since each entry of $A$ is bounded by $1$, $\norm{A}_{\hs} \leq n$. Notice that $\norm{D_2}_{\op} = \max_i |\xi'_i|$. Fix $i\in[n]$. Conditional on $X_i$, $\{\xi'_{ij}\}_{j\neq i}$ are i.i.d. random variables. Moreover, $|\xi'_i|\leq 1$. Therefore by Hoeffding's inequality, for any $t > 0$,
\[
    \mathbb{P}(|\xi'_i|>t|X_i)\leq e^{-\frac{t^2}{2(n-1)}}.
\]
Since the right hand side does not depend on $X_i$, the above bound also holds unconditionally. Now, the exchangeability of the $\xi'_i$'s yields
\begin{align*}
    \mathbb{P}\left(\frac{1}{n}\norm{D_2}_{\op}>\sqrt{\frac{6\log n}{n}}\right) &\leq n\mathbb{P}(|\xi_1'|>\sqrt{6n\log n})\\
    &\leq ne^{-\frac{6n\log n }{2(n-1)}}\\
    &\leq \frac{1}{n^2}.
\end{align*}
It follows by the Borel-Cantelli Lemma that $\frac{1}{n}\norm{D_2}_{\op} \convas 0$. We conclude that
\[
    d_{W_2}(\mu_M, \mu_{\tilde{M}}) \convas 0.
\]
This completes the proof.
\end{proof}
In fact, a further simplification is possible.
\begin{lemma}\label{lem:non_smooth_reduction_2}
    Suppose that $\bar{M} := \frac{1}{n}\sum_{i=1}^n \xi_i X_i X_i^\top$. Then $d_{W_2}(\mu_{\tilde{M}}, \mu_{\bar{M}}) \convas 0$.
\end{lemma}
\begin{proof}
Since $\tilde M= \frac{n-1}{n^2} \sum_{i=1}^n \xi_i X_i X_i^\top$, we have
\begin{equation}\label{eq:dist_Mtilde_Mbar}
    d_{W_2}(\mu_{\tilde M} , \mu_{\bar M}) \leq \norm{{\tilde M} - {\bar M}}_{\op} \leq \frac{1}{n^2} \norm{X X^\top}_{\op} \convas 0,
\end{equation}
where for the second inequality we have used the fact that $|\xi_i| \le 1$ so that the matrix
\[
    XX^\top - \sum_{i = 1}^n \xi_i X_i X_i^\top = \sum_{i = 1}^n (1 - \xi_i) X_i X_i^\top
\]
is positive semi-definite.
\end{proof}

By virtue of Lemmas~\ref{lem:non_smooth_reduction_1} and \ref{lem:non_smooth_reduction_2}, it is enough to find the LSD of $\bar{M}$. Let
\[
    S_n(z)\coloneqq \frac{1}{p} \Tr \left ({\bar M}-zI\right)^{-1}
\]
be the Stieltjes transform of $\bar{M}$. We shall now show that $S_n(z)$ converges almost surely as $n\to\infty$, for each fixed $z \in \mathbb{C}^+$. The following lemma shows that it suffices to find the limit of $\bbE S_n(z)$.

\begin{lemma}
For each fixed $z \in \mathbb{C}^+$,
\begin{equation}\label{eq:Stielt_a.s.=expected}
    S_n(z) - \bbE S_n(z) \convas 0.
\end{equation}
\end{lemma}
\begin{proof}
The proof uses the well-known martingale technique in random matrix theory (see, e.g., the proof of Theorem~3.10 of \cite{bai2010spectral}). Define $\mathcal{F}_0 = \{\phi, \Omega\}$ and $\mathcal{F}_k = \sigma(X_1, X_2, \ldots, X_k)$ for $k \in [n]$. By $\bbE_k$, $k= 0, 1, \ldots, n$, we denote the conditional expectation operator given $\mathcal{F}_k$. Then
\[
    S_n(z) - \bbE S_n(z) = \sum_{k=0}^n \bigg(\bbE_k \bigg[\frac{1}{p} \Tr(\bar{M} - zI)^{-1}\bigg] - \bbE_{k-1} \bigg[\frac{1}{p} \Tr(\bar{M} - zI)^{-1}\bigg]\bigg).
\]
Call the $k$-th summand above $\gamma_k$. Then $\{(\gamma_k, \mathcal{F}_k)\}_{k=1}^n$ is a martingale difference sequence. Suppose
\[
    \bar{M}_k = \bar{M} - \frac{1}{n}\xi_k X_k X_k^\top.
\]
Notice that
\[
    \bbE_k \Tr(\bar{M}_k - zI)^{-1}  = \bbE_{k-1} \Tr(\bar{M}_{k} - zI)^{-1}.
\]
By Lemma~\ref{lem:matrix_perturb}(b),
\begin{equation}\label{eq:Stieltjes_mg_diff_bound}
    |\Tr(\bar M - zI)^{-1} - \Tr(\bar M_k - zI)^{-1}| \leq \frac{1}{v}.
\end{equation}
Define $S_{nk}=\frac{1}{p}\Tr(\bar{M}_k - zI)^{-1}$. Therefore
\begin{equation}\label{eq:bdd_mg_diff}
    |\gamma_k| = |\bbE_k [S_n(z) - S_{nk}(z)] -\bbE_{k-1} [S_n(z) - S_{nk}(z)]| \leq 2/v.
\end{equation}
Thus $\gamma_k$ is a bounded martingale difference sequence. By Lemma~\ref{lem:Burkholder},
\begin{equation}\label{eq:Burkholder_application}
    \bbE|S_n(z) - \bbE S_n(z)|^4 \leq \frac{K_4}{p^4} \bbE \bigg( \sum_{k=1}^n |\gamma_k|^2 \bigg)^2 \leq \frac{4 K_4 n^2}{v^4 p^4} = O(n^{-2}).
\end{equation}
Now an application of the Borel-Cantelli lemma gives us \eqref{eq:Stielt_a.s.=expected}.
\end{proof}

\begin{lemma}\label{lemma:xi_conv}
    Suppose $V, W$ are i.i.d. copies of $X_1$. Define
    \[
        \zeta_p= \mathbb{E}[\phi_p(d(V,W))|V].
    \]
    Then $\zeta_p \convd \zeta$, where
    \[
        \zeta = \bbE_{Z_2} \bigg[\tilde{\phi}\bigg(\sqrt{\frac{m_2^{(1)}}{m_2}}Z_1 + \sqrt{\frac{m_2^{(2)}}{m_2}}Z_2\bigg) \bigg],
    \]
    where $Z_1, Z_2$ are i.i.d. $N(0,1)$ random variables.
\end{lemma}

\begin{proof}
Write $V=(V_1,V_2,\ldots,V_p)^\top$ and $W=(W_1,W_2,\ldots,W_p)^\top$. Let us fix the following notations:
\begin{align*}
    m_1'(v) &= \mathbb{E}[d(V_1,W_1)|V_1=v], \\
    m_2'(v)&= \text{Var}(d(V_1,W_1)|V_1=v),\\
    m_3'(v) &= \mathbb{E}[|d(V_1,W_1)-m_1'(V_1)|^3|V_1=v].
\end{align*}
Define
\[
    T=\frac{\sum_{i=1}^p d(V_i,W_i)-\sum_{i=1}^{p} m_1'(V_i)}{\sqrt{\sum_{i=1}^{p} m_2'(V_i)}}.
\]
Applying the Berry-Esseen theorem (see, e.g., \cite{bhattacharya1986normalapprox}) on $\{d(V_i,W_i)\}_{i=1}^p$, conditional on $V$,
\[
    \sup_x|\mathbb{P}(T\leq x|V)-\Phi(x)|\leq C_1 \frac{Q}{\sqrt{p}},
\]
where $C_1$ is an absolute constant and
\[
    Q = \bigg(\frac{1}{p}\sum _{i=1}^p m_2'(V_i)\bigg)^{-\frac{3}{2}} \bigg(\frac{1}{p}\sum_{i=1}^p m_3'(V_i)\bigg)^{\frac{1}{2}}.
\]
By the SLLN,
\[
    Q \convas {m_2^{(2)}}^{-\frac{3}{2}} (m_3)^{\frac{1}{2}}.
\]
For notational convenience, let us define
\begin{align*}
a_p(V)&=\sum_{i=1}^{p} m_1'(V_i),\\
b_p(V)&= \sqrt{\sum_{i=1}^{p} m_2'(V_i)}.
\end{align*}
Suppose $Z\sim N(0,1)$ is independent of $V$ and $W$. Without loss of generality assume that $\phi_p$ is increasing. Then,
\begin{align*}
\bigg|\mathbb{E}[\phi_p &(d(V,W))|V] - \mathbb{E}[\phi_p(a_p(V)+b_p(V)Z)|V]\bigg|\\
&= \bigg|\int_0^1\mathbb{P}(\phi_p(a_p(V)+b_p(V)T)>t|V) dt -\int_0^1\mathbb{P}(\phi_p(a_p(V)+b_p(V)Z)>t|V) dt\bigg|\\
&\leq \int_0^1 \bigg|\mathbb{P}\bigg(T>\frac{\phi_p^{-1}(t)-a_p(V)}{b_p(V)}\bigg|V\bigg)-\mathbb{P}\bigg(Z>\frac{\phi_p^{-1}(t)-a_p(V)}{b_p(V)}\bigg|V\bigg)\bigg| dt\\
&\leq C_1 \frac{Q}{\sqrt p} \to 0
\end{align*}
almost surely. Notice that
\[
    \mathbb{E}[\phi_p(a_p(V)+b_p(V)Z)|V]=\frac{1}{\sqrt{2 \pi}}\int_\mathbb{R}\phi_p(a_p(V)+b_p(V)t)e^{-\frac{t^2}{2}}dt.
\]
From the relation between $\phi$ and $\tilde{\phi}$, we have
\begin{align*}
    \phi_p(a(V)+b(V)t) &=\tilde{\phi}_p\circ \psi_p^{-1}(a(V)+b(V)t)\\
    &= \tilde{\phi}_p \left(\frac{a(V)+b(V)t-pm_1}{ \sqrt{pm_2}}\right)  \\
    &= \tilde{\phi}_p\left( a_p'(V)+b_p'(V)t\right),
\end{align*}
where
\begin{align*}
    a_p'(V)&= \frac{a(V)-pm_1}{\sqrt{pm_2}},\\
    b_p'(V) &= \frac{b(V)}{\sqrt{pm_2}}.
\end{align*}
Using the central limit theorem on $\{m_1'(V_i)\}_{i=1}^p$, as $p\to \infty$, we have
\[
    \frac{a_p(V)-pm_1}{\sqrt{pm_2^{(1)}}} \convd N(0,1),
\]
and consequently,
\[
    a_p'(V) \convd N\left(0,\frac{m_2^{(1)}}{m_2}\right).
\]
On the other hand, an application of the SLLN shows that
\[
b_p'(V) \convas \sqrt{\frac{m_2^{(2)}}{m_2}}.
\]
We are interested in the weak limit of
\[
    \frac{1}{\sqrt{2 \pi}}\int_\mathbb{R}\tilde{\phi}_p(a_p'(V)+b_p'(V)t)e^{-\frac{t^2}{2}}dt
\]
as $p \to \infty$. Suppose (by an abuse of notation) that $a_p'$, $b_p'$ are sequences of real numbers. Fix $\epsilon>0$ and define the event
\[
    G = \bigg\{\bigg|\tilde{\phi}_p(a_p'+b_p't)dt - \tilde{\phi}(a_p'+b_p't)\bigg|\leq \epsilon\bigg\}.
\]
Then
\begin{align*}
    \int_\mathbb{R}|\tilde{\phi}_p &(a_p'+b_p't)dt - \tilde{\phi}(a_p'+b_p't)|e^{-\frac{t^2}{2}}dt\\
    &\leq \int_G|\tilde{\phi}_p(a_p'+b_p't)dt - \tilde{\phi}(a_p'+b_p't)|e^{-\frac{t^2}{2}}dt+\int_{G^c}|\tilde{\phi}_p(a_p'+b_p't)dt - \tilde{\phi}(a_p'+b_p't)|e^{-\frac{t^2}{2}}dt.
\end{align*}
The first integral above is bounded by $\sqrt{2\pi}\epsilon$ and the second one by $\Leb(|\tilde{\phi}_p-\tilde{\phi}|>\epsilon)$. Taking a supremum over $a_p'$, $b_p'$, we get
\[
    \sup_{a_p',\,b_p'}\int_\mathbb{R}|\tilde{\phi}_p(a_p'+b_p't)dt - \tilde{\phi}(a_p'+b_p't)|e^{-\frac{t^2}{2}}dt \leq \sqrt{2\pi}\epsilon + \Leb(|\tilde{\phi}_p-\tilde{\phi}|>\epsilon).
\]
Now, sending $p\to \infty$ and noticing that $\epsilon$ is arbitrary, we obtain
\[
    \sup_{a_p',\,b_p'}\int_\mathbb{R}|\tilde{\phi}_p(a_p'+b_p't)dt - \tilde{\phi}(a_p'+b_p't)|e^{-\frac{t^2}{2}}dt \to 0.
\]
as $p\to \infty$. It is enough to consider the weak limit of
\[
    \frac{1}{\sqrt{2 \pi}}\int_\mathbb{R}\tilde{\phi}_p(a_p'(V)+b_p'(V)t)e^{-\frac{t^2}{2}}dt
\]
as $p\to \infty$.

Next we prove that if $a_p'\to a'$ and $b_p' \to b'$ (both deterministic sequences) as $p \to \infty$, then
\[
\int_\mathbb{R}\tilde{\phi}(a_p'+b_p't)e^{-\frac{t^2}{2}}dt \to \int_\mathbb{R}\tilde{\phi}(a'+b't)e^{-\frac{t^2}{2}}dt,
\]
in other words, the map $(a',b')\mapsto \int_\mathbb{R}\tilde{\phi}(a'+b't)e^{-\frac{t^2}{2}} $ is continuous. Since $\tilde{\phi}$ is monotonic, it is continuous almost everywhere. Thus, $\tilde{\phi}(a_p+b_pt)\to \tilde{\phi}(a+bt)$ for almost every $t$. An application of the bounded convergence theorem then proves our claim. Combining this with the convergence of $a'(V)$ and $b'(V)$, we get
\[
    \frac{1}{\sqrt{2\pi}}\int_\mathbb{R}\tilde{\phi}(a_p'(V)+b_p(V)'t)e^{-\frac{t^2}{2}}dt \overset{\text{d}}\longrightarrow \frac{1}{\sqrt{2 \pi}}\int_\mathbb{R}\tilde{\phi}\left(\sqrt{\frac{m_2^{(1)}}{m_2}}Z+ \sqrt{\frac{m_2^{(2)}}{m_2}}t\right)e^{-\frac{t^2}{2}}dt,
\]
where $Z\sim N(0,1)$. This completes the proof.
\end{proof}

\begin{lemma}\label{lem:approx_recur}
$\mathbb{E} S_n(z)$ satisfies the following approximate recursion:
\[
    1 + z\mathbb{E}S_n(z) = \mathbb{E}_{\zeta}\bigg[\frac{\sigma^2\zeta\mathbb{E}S_n(z)}{1+c_n\sigma^2\zeta\mathbb{E}S_n(z)}\bigg] + o(1).
\]
\end{lemma}
\begin{proof}
    Suppose $\xi_{0}\sim \xi_1$, $X_{0}\sim X_1$ are mutually independent and they are independent of the $\{X_i\}_{i=1}^n$. Define
\begin{align*}
    \bar{M}'&=\bar{M}+\frac{1}{n}\xi_{0}X_{0}X_{0}^\top,\\
    S_n'(z)&= \frac{1}{p}\Tr\left ({\bar M}'-zI\right)^{-1}.
\end{align*}
Then
\begin{align*}
    p &= \Tr\left ({\bar M}'-zI\right) \left ({\bar M}'-zI\right)^{-1}\\
    &=\frac{1}{n}\sum_{i=0}^{n} \xi_i X_i^\top \left ({\bar M}'-zI\right)^{-1} X_i- z \Tr \left ({\bar M}'-zI\right)^{-1}.
\end{align*}
Taking expectation on both sides and using the exchangeability of the summands, we have
\begin{align*}
    p&=\left(1+\frac{1}{n}\right) \mathbb{E}(\xi_0 X_0^\top \big ({\bar M}'-zI\big)^{-1}X_0) - pz\mathbb{E}S_n'(z).
\end{align*}
Dividing by $n$ and noticing that $|\xi_0 X_0^\top \big ({\bar M}'-zI\big)^{-1}X_0|\leq 1+\frac{|z|}{v}$, we get
\[
    c_n+ c_nz\mathbb{E}S'_n(z)=\frac{1}{n}\mathbb{E}(\xi_0 X_0^\top\big ({\bar M}'-zI\big)^{-1}X_0)+o(1).
\]
Observe that
\[
    |S'_n(z)-S_n(z)|\leq \frac{v}{p}.
\]
Therefore
\[
    c_n+c_nz\mathbb{E}S_n(z)=\frac{1}{n}\mathbb{E}(\xi_0X_0^\top\big ({\bar M}'-zI\big)^{-1}X_0) +o(1).
\]
An application of the Sherman-Morrison formula on the first term yields
\[
    \frac{1}{n}\xi_0X_0^\top\big ({\bar M}'-zI\big)^{-1}X_0= \frac{\frac{1}{n}\xi_0X_0^\top\left ({\bar M}-zI\right)^{-1}X_0}{1+\frac{1}{n}\xi_0X_0^\top\left ({\bar M}-zI\right)^{-1}X_0}.
\]
Then
\begin{align*}
    \bigg|\frac{1}{n}\xi_0 &X_0^\top\big ({\bar M}'-zI\big)^{-1}X_0-\frac{c_n\sigma^2\xi_0\mathbb{E}S_n(z)}{1+c_n\sigma^2\xi_0\mathbb{E}S_n(z)}\bigg|\\
    &=\bigg|\frac{\frac{1}{n}\xi_0X_0^\top\big ({\bar M}-zI\big)^{-1}X_0}{1+\frac{1}{n}\xi_0X_0^\top\big ({\bar M}-zI\big)^{-1}X_0}-\frac{\frac{1}{n}\xi_0\mathbb{E}(X_0^\top\big ({\bar M}-zI\big)^{-1}X_0)}{1+\frac{1}{n}\xi_0\mathbb{E}(X_0^\top\big ({\bar M}-zI\big)^{-1}X_0)}\bigg|\\
    &\leq \frac{\frac{1}{n}\xi_0|X_0^\top\big ({\bar M}-zI\big)^{-1}X_0-\mathbb{E}(X_0^\top\big ({\bar M}-zI\big)^{-1}X_0)|}{|1+\frac{1}{n}\xi_0X_0^\top\big ({\bar M}-zI\big)^{-1}X_0||1+\frac{1}{n}\xi_0\mathbb{E}(X_0^\top\big ({\bar M}-zI\big)^{-1}X_0)|} \\
    &\leq \frac{\frac{|z|^2}{n}|X_0^\top\big ({\bar M}-zI\big)^{-1}X_0-\mathbb{E}(X_0^\top\big ({\bar M}-zI\big)^{-1}X_0)|}{|z+z\frac{1}{n}\xi_0X_0^\top\big ({\bar M}-zI\big)^{-1}X_0||z+z\frac{1}{n}\xi_0\mathbb{E}(X_0^\top\big ({\bar M}-zI\big)^{-1}X_0)|} \\
    &\leq \frac{|z|^2}{v^2} \frac{1}{n}|X_0^\top\big ({\bar M}-zI\big)^{-1}X_0-\mathbb{E}(X_0^\top\big ({\bar M}-zI\big)^{-1}X_0)|,
\end{align*}
where we have used Lemma~\ref{lem:matrix_perturb}(h) to lower bound the denominator. Since $X_0$ is independent of $(\bar M -zI)^{-1}$, by Lemma~\ref{lem:matrix_quad_form_moment},
\[
    \bbE |X_0^\top\left ({\bar M}-zI\right)^{-1}X_0-\mathbb{E}(X_0^\top\left ({\bar M}-zI\right)^{-1}X_0)| \leq C\bbE \sqrt{\Tr(\bar M -zI)^{-1}} \leq \frac{C \sqrt p}{\sqrt v},
\]
for some constant $C > 0$, which only depends on the second and fourth moments of $w_{11}$.

Thus
\[
    \mathbb{E}\frac{1}{n}\xi_0X_0^\top\big ({\bar M}'-zI\big)^{-1}X_0=\mathbb{E}\frac{c_n\sigma^2\xi_0\mathbb{E}S_n(z)}{1+c_n\sigma^2\xi_0\mathbb{E}S_n(z)}+o(1).
\]
Also,
\begin{align*}
    \bigg|\mathbb{E}\frac{c_n\sigma^2\xi_0\mathbb{E}S_n(z)}{1+c_n\sigma^2\xi_0\mathbb{E}S_n(z)} &-\mathbb{E}\frac{c_n\sigma^2\zeta\mathbb{E}S_n(z)}{1+c_n\sigma^2\zeta\mathbb{E}S_n(z)}\bigg|\\
    &\leq \sigma^2 c_n|z|^2|\mathbb{E}S_n(z)| \mathbb{E}\frac{|\xi_0-\zeta|}{|z+zc_n\sigma^2\xi_0\mathbb{E}S_n(z)||z+zc_n\sigma^2\zeta_0\mathbb{E}S_n(z)|}\\
    &\leq \sigma^2 c_n \frac{|z|^2}{v^2}\frac{1}{v}\mathbb{E}|\xi_0-\zeta|.
\end{align*}
Lemma~\ref{lemma:xi_conv} coupled with the Skorohod representation theorem gives us that $\mathbb{E}|\xi_0-\zeta|\to 0$ as $p\to\infty$. Combining everything, we get the desired approximate functional for $\bbE S_n(z)$:
\[
    c_n+c_nz\mathbb{E}S_n(z)=\mathbb{E}_{\zeta}\bigg[\frac{c_n\sigma^2\zeta\mathbb{E}S_n(z)}{1+c_n\sigma^2\zeta\mathbb{E}S_n(z)}\bigg] + o(1).
\]
This completes the proof.
\end{proof}

\begin{lemma}\label{lem:sol_uniqueness}
The equation
\[
    1+zs(z)=\mathbb{E}_{\zeta}\bigg[\frac{\sigma^2\zeta s(z)}{1+c\sigma^2\zeta s(z)}\bigg],
\]
$z\in\mathbb{C}^+$, has a unique solution for $s(z)$ in $\mathbb{C}^+$, where $\zeta$ is defined in \eqref{eq:zeta_defn}.
\end{lemma}
\begin{proof}
Suppose \eqref{eq:non_smooth_ST} has two distinct solutions $s_1, s_2$ in $\mathbb{C}^+$. Fix some $z\in \mathbb{C}^2$ such that $s_1(z) \neq s_2(z)$. Note for $j=1, 2$,
\[
    s_j = \frac{1}{-z + \bbE \frac{\sigma^2 \zeta}{1 + c \sigma^2 s_j \zeta}}.
\]
Notice
\[
    \mathrm{Im}(s_j) = \frac{v - \mathrm{Im}(\bbE \frac{\sigma^2 \zeta}{1 +c \sigma^2 s_j \zeta})}{|-z + \bbE \frac{\sigma^2 \zeta}{1 + c \sigma^2 s_j \zeta}|^2} < \frac{\bbE\bigg[\frac{ c \sigma^4 \mathrm{Im}(s_j) \zeta^2}{|1 + c \sigma^2 s_j \zeta|^2}\bigg]}{|-z + \bbE \frac{\sigma^2 \zeta}{1 + c \sigma^2 s_j \zeta}|^2}.
\]
So,
\begin{equation}\label{eq:lem_4.2_aux}
    \frac{\bbE\big[\frac{ c \sigma^4 \zeta^2}{|1 + c \sigma^2 s_j \zeta|^2}\big]}{|-z + \bbE \frac{\sigma^2 \zeta}{1 + c \sigma^2 s_j \zeta}|^2} > 1.
\end{equation}
Also,
\[
    s_1 - s_2 = \frac{c \sigma^4 (s_1 - s_2) \bbE \frac{\zeta^2}{(1 + c \sigma^2 s_1 \zeta)(1 + c \sigma^2 s_2 \zeta)} }{(-z + \bbE \frac{\sigma^2 \zeta}{1 + c \sigma^2 s_1 \zeta})(-z + \bbE \frac{\sigma^2 \zeta}{1 + c \sigma^2 s_2 \zeta})}.
\]
Cancelling $s_1 - s_2$ from both sides and applying the Cauchy-Schwarz inequality on the right hand side we get
\[
    1 \leq \frac{\bbE\big[\frac{ c \sigma^4 \zeta^2}{|1 + c \sigma^2 s_1 \zeta|^2}\big]}{|-z + \bbE \frac{\sigma^2 \zeta}{1 + c \sigma^2 s_1 \zeta}|^2} \frac{\bbE\big[\frac{ c \sigma^4 \zeta^2}{|1 + c \sigma^2 s_2 \zeta|^2}\big]}{|-z + \bbE \frac{\sigma^2 \zeta}{1 + c \sigma^2 s_2 \zeta}|^2},
\]
which contradicts \eqref{eq:lem_4.2_aux}.
\end{proof}

Now we are ready to prove Theorem~\ref{thm:conv_nonsmooth}.
\begin{proof}[Proof of Theorem~\ref{thm:conv_nonsmooth}]
Applying Lemmas~\ref{lem:non_smooth_reduction_1} and \ref{lem:non_smooth_reduction_2}, it is enough to consider $\bar{M}$ instead of $M$.  Since $|\mathbb{E}S_n(z)|\leq \frac{1}{v}$, using the Bolzano-Weierstrass theorem, $\mathbb{E}S_n(z)$ has a convergent subsequence. Consider any such subsequence $|\mathbb{E}S_{n_k}(z)|\to s(z)$. Then, by Lemma~\ref{lem:approx_recur}, $s(z)$ will satisfy the limiting equation
\[
1+zs(z)=\mathbb{E}\frac{\sigma^2\zeta s(z)}{1+c\sigma^2\zeta s(z)}.
\]
Lemma~\ref{lem:sol_uniqueness} shows that this equation has a unique solution for $s(z)$ in $\mathbb{C}^+$.

This shows that all the subsequential limits of $\mathbb{E}S_n(z)$ are same proving that $\mathbb{E}S_n(s)\to s(z)$ as $n\to \infty$ where $s(z)$ is the unique solution in $\mathbb{C}^+$ of the aforementioned equation. We can thus conclude that, almost surely, $\mu_{\bar{M}}$ converges weakly to some sub-probability measure $\mu$ with Stieltjes transform $s$. In order to show that $\mu$ is indeed a probability measure, by Prohorov's theorem it is sufficient to show that $\int x^2 \mathbb{E}\mu_{\bar{M}}(dx)$ is uniformly bounded. By the trace-moment formula, the last quantity is the same as $\frac{1}{p}\mathbb{E}\Tr(\bar{M}^2)$ which we compute next.
\begin{align*}
    \frac{1}{p}\mathbb{E}\Tr(\bar{M}^2) &= \frac{1}{pn^2} \bbE\Tr \bigg( \sum_{i=1}^n \xi_i^2 X_iX_i^\top X_i X_i^\top +\sum_{i \neq j} \xi_i \xi_j X_i X_i^\top X_j X_j ^\top \bigg)\\
    &= \frac{1}{p n^2} (n\bbE \norm{X_1}^4 + n(n-1) \bbE (X_1^\top X_2)^2)\\
    &= O(1),
\end{align*}
where the last step follows from the fact that $\bbE \norm{X_1}^4 = O(p^2)$ and $\bbE (X_1^\top X^2)^2 = \V(X_1^\top X_2) = O(p)$. This completes the proof of Theorem~\ref{thm:conv_nonsmooth}.
\end{proof}

\subsection{Proof of Theorem~\ref{thm:conv_smooth}}
We again begin with the matrix Rayleigh quotient representation:
\[
    M = \frac{1}{n^2} X^\top L X^\top.
\]

Observe that $\mathbb{E}[A] = \alpha_{n, p} (J - I)$ and $\mathbb{E}[D]= (n - 1) \alpha_{n, p} I$, where $J = \mathbf{1}\mathbf{1}^\top$. Define
\[
    \cL = \bbE[L] = \bbE[D] - \bbE[A] = \alpha_{n, p} ((n - 1) I - (J - I)) = \alpha_{n, p} (nI - J).
\]

Further, let
\begin{align}
    F &=  \frac{1}{n} (A - \alpha_{n, p} (J - I)), \\
    G &= \frac{1}{n} (D - (n - 1)\alpha_{n, p} I) = \diag(F \bone).
\end{align}
With the above notation, $\frac{1}{n}L = \frac{1}{n} \cL - F + G$. Thus
\begin{align*}
    M = \frac{1}{n^2} X L X^\top &= \frac{1}{n} X \bigg(\frac{1}{n} \cL + F - G\bigg) X^\top \\
    &= \frac{1}{n^2} X\cL X^\top + \frac{1}{n} X (F - G) X^\top \\
    &= \tilde{M} + \frac{1}{n} X (F - G) X^\top,
\end{align*}
where $\tilde{M}= \frac{1}{n^2} X \cL X^\top$. We will show that the $M$ and $\tilde{M}$ have the same LSD. Indeed,
\begin{align*}
    d_{W_2}(\mu_M, \mu_{\tilde M}) &\leq \frac{1}{\sqrt{n}}\norm{M - \tilde{M}}_{\hs} \\
    &= \frac{1}{n \sqrt n} \norm{X (F - G) X^\top}_{\hs} \\
    &\leq \frac{1}{n \sqrt n} \norm{X}_{\op} \norm{F - G}_F \norm{X^\top}_{\op} \\
    &= \frac{1}{n}\norm{X X^\top}_{\op} \frac{1}{\sqrt n} \norm{F - G}_{\hs} \\
    &\leq \frac{1}{n} \norm{X X^\top}_{\op} \frac{1}{\sqrt n}(\norm{F}_{\hs} + \norm{G}_{\hs}) \\
    &\leq \frac{1}{n} \norm{X X^\top}_{\op} (\norm{F}_{\op} + \frac{1}{\sqrt{n}} \norm{F\mathbf{1}}) \\
    &\leq \frac{2}{n} \norm{X X^\top}_{\op} \norm{F}_{\op}.
\end{align*}

Under the finite fourth moment condition, recall that $\frac{1}{n}\norm{X X ^\top}_{\op}\to \sigma^2(1 + \sqrt{c})^2$ almost surely. We now show that $\norm{F}_{\op} \convas 0$. Because of Assumption~\ref{ass:lc_class}, we can apply Theorem~1 of \cite{amini2021concentration} on $A$ to get that
\[
    \P\bigg(\norm{F}_{\op} \geq 2 L \omega \sigma\bigg(C + \frac{t}{\sqrt{n}}\bigg)\bigg) \leq \exp(-t^2 / C^2)
\]
for some $C > 0$. We can simplify this as follows:
\[
    \P\bigg(\norm{F}_{\op} \geq 2 L \omega \sigma t\bigg) \leq \exp(- C'^2 t^2 ).
\]
for some $C'>0$. Now choosing $t_n= \frac{1}{C'} \sqrt{\log n}$, we note that
\[
    \P\bigg(\norm{F}_{\op} \geq \frac{2 L \omega \sigma \sqrt {\log n}}{C'} \bigg) \leq \frac{1}{n^2}.
\]
Combining this with the fact that $L \omega = o(1/\sqrt{\log n})$, we conclude that $\norm{F}_{\op} \to 0$ almost surely. This implies that $d_{W_2}(\mu_M, \mu_{\tilde{M}})\to 0$ almost surely, that is $M$ and $\tilde{M}$ have the weak limit almost surely if it exists. Notice that $\mathcal{L} = \alpha_{p}n I - \alpha_{p} J$. Since, $J$ is of rank 1, by Lemma~\ref{lem:rank_ineq},
\[
    d_{KS}(\mu_{\tilde{M}},\mu_{\hat{M}}) \leq \frac{1}{n}
\]
where $d_{KS}$ denotes the Kolmogorov-Smirnov distance and $\hat{M} = \frac{\alpha_{p}}{n} XX^\top.$ Since
\[
    d_{W_2}(\mu_{\hat{M}}, \mu_{\frac{\alpha}{n}X X^\top}) \leq \frac{|\alpha_{p}-\alpha|}{n}\norm{X X^\top}_\op \xrightarrow{a.s.} 0,
\]
we can only consider $\frac{\alpha_{p}}{n}X X^\top$. It is well known that the last matrix has $\MP_{c, \alpha^2\sigma^2}$ as its LSD. Combining everything, we conclude that
\[
\mu_M \xrightarrow{d} \MP_{c, \alpha^2\sigma^2},
\]
almost surely. This completes the proof.

\subsection{Proof of Theorem~\ref{thm:conv_semi_high_dim}}

\begin{proof}[Proof of Theorem~\ref{thm:conv_semi_high_dim}]We define $A, D, L$ as in the proof of Theorem~\ref{thm:conv_semi_high_dim} and
\begin{align*}
    &{\tilde E} \coloneqq \sqrt{\frac{n}{p}} \bigg(\frac{n-1}{n^2} \sum_{i=1}^n \xi_i X_i X_i^\top - \alpha_p \sigma^2 I\bigg), &{\bar E} &\coloneqq \sqrt{\frac{n}{p}} \bigg(\frac{1}{n} \sum_{i=1}^n \xi_i X_i X_i^\top - \alpha_p \sigma^2 I\bigg).
\end{align*}
One can prove analogues of Lemmas~\ref{lem:non_smooth_reduction_1} and \ref{lem:non_smooth_reduction_2} in this set up so that $d_{W_2}(\mu_E, \mu_{\tilde E}) \convas 0$ and $d_{W_2}(\mu_{\tilde E} , \mu_{\bar E}) \convas 0$, respectively (the analogue of the Bai-Yin result in this regime can be found in \cite{chen2012convergence}). For the purpose of finding the LSD, it is thus enough to consider $\bar E$. Now, $\bar{E}$ can be further decomposed as $\bar{E}_1 + \bar{E}_2$, where
\begin{align*}
    \bar{E}_1 = \frac{1}{\sqrt{np}} \sum_{i=1}^n \xi_i (X_i X_i^\top - \sigma^2 I), \qquad \text{and} \qquad
    \bar{E}_2 = \frac{\sigma^2}{\sqrt{np}} \sum_{i=1}^n (\xi_i - \alpha_p) I.
\end{align*}
Since the $(\xi_i - \alpha_p)$'s are centered, independent and bounded, invoking Hoeffding's inequality, $\norm{\bar{E}_2}_{\op} \leq \frac{\sigma^2}{\sqrt[4]{p}}$ w.p. $\geq 1 - 2e^{-\frac{\sqrt{p}}{2}}$, implying that $d_{W_2} (\mu_{\bar{E}}, \mu_{\bar{E}_1}) \convas 0$. In order to find the LSD of $\bar E_1$, we proceed with the Stieltjes transform once again. We first show \eqref{eq:Stielt_a.s.=expected}. Our technique remains same for this part, except that we need analogues of equations \eqref{eq:bdd_mg_diff} and \eqref{eq:Burkholder_application}. First we define
\begin{align}
    &\tilde S_n(z) = \frac{1}{p} \Tr( \bar{E}_1 - zI)^{-1}, \label{eq:ST_defn_semi_high}\\
    &\bar{E}_{1k} = \bar{E}_1 -\frac{1}{\sqrt{np}}\xi_k (X_k X_k^\top - \sigma^2 I).
\end{align}
By Lemma~\ref{lem:matrix_perturb} (b) and (c),
\[
    |\Tr(\bar E_1 - zI)^{-1} -\Tr(\bar E_{1k} - zI)^{-1}| \leq \frac{p + \norm{X_k}^2}{v^2 \sqrt{np}}.
\]
Now, by virtue of Lemma~\ref{lem:Burkholder} for $l=4$,
\begin{align*}
    \bbE|\tilde{S}_n(z) - \bbE \tilde{S}_n(z)|^4 &\leq \frac{K_4}{p^4} \bbE \bigg( \sum_{k=1}^n \bigg(\frac{2(p + \norm{X_k}^2)}{\sqrt{np}}\bigg)^2 \bigg)^2 \\
    &\leq \frac{16 K_4 n^2}{v^4 n^2 p^6} \bbE \bigg( \sum_{k=1}^n (2p^2 + 2\norm{X_k}^4)\bigg)^2 \\
    &\leq \frac{64 K_4 n^2}{v^4 n^2 p^6} \bbE \bigg(np^2 + \sum_{k=1}^n \norm{X_k}^4\bigg)^2 \\
    &\leq \frac{128 K_4 n^2}{v^4 n^2 p^6}  \bigg(n^2p^4 + \bbE \bigg(\sum_{k=1}^n \norm{X_k}^4\bigg)^2\bigg) \\
    &\leq \frac{128 K_4 n^2}{v^4 n^2 p^6}  \bigg(n^2p^4 + n\bbE \norm{X_1}^8 + n(n-1) (\bbE \norm{X_1}^4)^2 \bigg) \\
    &=O(1/p^2).
\end{align*}
Here the last line follows from the facts that $\bbE \norm{X_1}^8 = O(p^4)$ and $\bbE \norm{X_1}^4 = O(p^2)$. Now, by Borel-Cantelli lemma, \eqref{eq:Stielt_a.s.=expected} follows. So, it is enough to consider $\bbE \tilde{S}_n(z)$. Define $\tilde{s}_n(z) = \bbE \tilde{S}_n(z)$. By Lemma~\ref{lem:approx_recur_semi_high}, we get
\[
    \sigma^4 \beta_n^2 \tilde{s}_n^2(z) + z \tilde{s}_n(z) + 1 = o(1),
\]
which has the solutions
\begin{equation}\label{eq:Stielt_trans_quad}
    \tilde{s}_n(z) = \frac{-z \pm \sqrt{z^2 - 4\sigma^4 \beta_n^2 + o(1)}}{2\sigma^4 \beta_n^2}.
\end{equation}
Here $\sqrt {z'}$, for any $z' \in \mathbb{C}\backslash \mathbb{R}$, denotes the square root in the upper-half plane. Clearly, one must take the $+$ sign in  \eqref{eq:Stielt_trans_quad}, otherwise the right hand side has negative imaginary part which is not allowed. Now taking limit as $n \to \infty$, we get
\[
    \lim_{n \to \infty} \tilde{s}_n(z) = \frac{-z \pm \sqrt{z^2 - 4\sigma^4 \beta^2}}{2\sigma^4 \beta^2},
\]
which the Stieltjes transform of the semi-circle law with variance $\beta^2 \sigma^4$. This completes the proof of Theorem~\ref{thm:conv_semi_high_dim}.
\end{proof}
\begin{lemma}\label{lem:approx_recur_semi_high}
   Let $\tilde{S}_n(z)$ and $\tilde{s}_n(z)$ be defined as in Theorem~\ref{thm:conv_semi_high_dim}. Then, $\tilde{s}_n(z)$ satisfies the following approximate functional equation
   \[
   \sigma^4 \beta_n^2 \tilde{s}_n^2(z) + z \tilde{s}_n(z) + 1 = o(1).
   \]
\end{lemma}
\begin{proof}
To this end, take an independent copy $(\xi_0, X_0)$ of $(\xi_1, X_1)$. Define \begin{align*}
    & \bar{E}'_1 = \frac{1}{\sqrt{np}} \sum_{i=0}^n \xi_i (X_i X_i^\top - \sigma^2 I),
    & S'_n = \frac{1}{p} \Tr( \bar{E}'_1 - zI)^{-1}.
\end{align*}
Proceeding as in the proof of Theorem~\ref{thm:conv_nonsmooth}, we can get
\begin{equation}\label{eq:semi_high_Stielt_trans}
    1 + z \tilde{s}'_n(z) = \frac{n+1}{p\sqrt{np}} \bbE \xi_0 (X_0^\top (\bar{E}'_1 - zI)^{-1} X_0 - \sigma^2 \Tr(\bar{E}'_1 - zI)^{-1}),
\end{equation}
where $\tilde{s}'_n(z) = \bbE \tilde{S}'_n(z)$. Since $\norm{\bar{E}'_1}_{\op} \leq \frac{1}{v}$,
\[
    |\bbE \xi_0 (X_0^\top (\bar{E}'_1 - zI)^{-1} X_0 - \sigma^2 \Tr(\bar{E}'_1 - zI)^{-1})| \leq \frac{1}{v} \bbE \norm{X_0}^2 + \frac{p \sigma^2 }{v} \leq \frac{2p \sigma^2 }{v}.
\]
Thus, for the asymptotic purposes one can change the $n+1$ by $n$ in the right hand side of \eqref{eq:semi_high_Stielt_trans}. By Lemma~\ref{lem:matrix_perturb}(b), (c),
\[
    |\Tr(\bar{E}'_1 - zI)^{-1} - \Tr(\bar{E}_1 - zI)^ {-1}| \leq \frac{1}{v^2} \bigg(\frac{1}{\sqrt{np}} \norm{X_0}^2 + \sigma^2 \sqrt{\frac{p}{n}} \bigg),
\]
where the expectation of the right hand side is $\frac{2 \sigma^2}{v^2} \sqrt{\frac{p}{n}}$. This also implies that $|\tilde{s}'_n(z) - \tilde{s}_n(z)| \leq \frac{2 \sigma^2}{v^2}\frac{1}{\sqrt{np}}$. Now, applying Lemma~\ref{lem:matrix_perturb}(b),
\[
    \bigg|X_0^\top (\bar{E}'_1 - zI)^{-1} X_0 - X_0^\top \bigg(\bar{E}_1 + \frac{1}{\sqrt{np}} \xi_0 X_0 X_0^\top - zI \bigg)^{-1} X_0 \bigg| \leq \frac{\sigma^2}{v \sqrt{np}} \norm{X_0}^2,
\]
where the right hand side has expectation $\frac{\sigma^2}{v} \sqrt{\frac{p}{n}}$. Using these perturbation results, we have the following upgrade of \eqref{eq:semi_high_Stielt_trans}:
\begin{equation}\label{eq:semi_high_Stielt_trans_2}
    1 + z\tilde{s}_n(z) = \frac{\sqrt{n}}{p\sqrt{p}} \bbE \xi_0 \bigg(X_0^\top \bigg( \bar{E}_1 + \frac{1}{\sqrt{np}} \xi_0 X_0 X_0^\top - zI \bigg)^{-1} X_0 - \sigma^2 \Tr(\bar{E}_1 - zI)^{-1}\bigg) + o(1).
\end{equation}
We apply Sherman-Morrison formula (see Lemma~\ref{lem:matrix_perturb}(i)) on the first term in the right hand side to get
\[
    \frac{1}{\sqrt{np}} \xi_0 X_0^\top \bigg( \bar{E}_1 + \frac{1}{\sqrt{np}} \xi_0 X_0 X_0^\top - zI \bigg)^{-1} X_0 = \frac{\frac{1}{\sqrt{np}} \xi_0 X_0^\top \big( \bar{E}_1 - zI \big)^{-1} X_0}{1 + \frac{1}{\sqrt{np}} \xi_0 X_0^\top \big( \bar{E}_1 - zI \big)^{-1} X_0}.
\]
Notice we can write \eqref{eq:semi_high_Stielt_trans_2} in the form $1 + z\tilde{s}_n(z) = - \bbE B_1 + \bbE B'_1 + \bbE B_2 - \bbE B'_2 +o(1)$, where
\begin{align*}
     B_1 &= \frac{\sigma^2}{p^2} \xi_0^2 \Tr(\bar{E}_1 - zI)^{-1} X_0^\top ( \bar{E}_1 - zI )^{-1} X_0,\\
     B'_1 &= \frac{B_1 \frac{1}{\sqrt{np}} \xi_0 X_0^\top ( \bar{E}_1 - zI )^{-1} X_0}{1 + \frac{1}{\sqrt{np}} \xi_0 X_0^\top ( \bar{E}_1 - zI )^{-1} X_0}, \\
     B_2 &= \frac{\sqrt n}{p \sqrt p} \xi_0 R,\\
     B'_2 &= \frac{1}{p^2} \frac{\xi_0^3 R X_0^\top (\bar{E}_1 -zI)^{-1} X_0}{1 + \frac{1}{\sqrt{np}} \xi_0 X_0^\top (\bar{E}_1 -zI)^{-1} X_0}, \\
     R &= X_0^\top (\bar{E}_1 - zI)^{-1} X_0 - \sigma^2 \Tr(\bar{E}_1 - zI)^{-1}.
\end{align*}
We show that among the four, only $B_1$ contributes to the equation. Note
\[
    B_1 = \frac{\sigma^4}{p^2} \xi_0^2 (\Tr(\bar{E}_1 - zI)^{-1})^2 + \frac{\sigma^2}{p^2} \xi_0^2 R\Tr(\bar{E}_1 - zI)^{-1} = \sigma^4 \xi_0^2 \tilde{S}_n^2(z)+\sigma^4 \tilde{S}_n(z) \frac{R}{p}.
\]
By using Lemma~\ref{lem:matrix_quad_form_moment},
\begin{equation}\label{eq:Hanson_Wright_moments}
    \bbE |R|^2 \leq C_2 \bbE\Tr\big((\bar{E}_1 - zI)^{-2}\big) \leq \frac{C_2 p}{v^2},
\end{equation}
where $C_2$ only depends on the fourth moment of $w_{11}$. This additionally shows that $R= O_P(1)$. Also, $\V(\tilde{S}_n(z)) \to 0$ because, $\tilde{S}_n(z)$ is a bounded sequence that converges almost surely. Thus,
\[
    \bbE B_1 = \sigma^4 \bbE(\xi_0^2 \tilde{S}_n^2(z)) + o(1) = \sigma^4 \bbE \xi_0^2 \bbE \tilde{S}_n^2(z) +o(1) =\sigma^4 \beta_n^2 \tilde{s}_n^2(z) + o(1).
\]
By \eqref{eq:Hanson_Wright_moments} and \eqref{ass:dimension}, $\bbE B_2 = o(1)$. Notice by Lemma~\ref{lem:matrix_perturb}(h), $Im(z + \frac{z}{\sqrt{np}} \xi_0 X_0^\top (\bar{E}_1 -zI)^{-1} X_0) \geq v$, $|X_0^\top (\bar{E}_1 -zI)^{-1} X_0| \leq \norm{X_0}^2/v$ and $\Tr(\bar{E}_1 -zI)^{-1} \leq p/v$. So,
\begin{align*}
    &\bbE |B'_1| \leq \frac{\sigma^2 |z| }{v^3 p \sqrt{np}} \bbE \norm{X_0}^4 = \frac{\sigma^2 |z|}{v^3 p \sqrt{np}} (p \bbE w_{11}^4 +p(p-1)\sigma^4) = o(1),\\
    & \bbE |B'_2| \leq \frac{|z|}{v^2 p^2} \bbE [|R| \norm{X_0}^2] \leq \frac{|z|}{v^2 p^2} (\bbE R^2)^{1/2} (\bbE \norm{X_0}^4)^{1/2} = o(1).
\end{align*}
Now, \eqref{eq:semi_high_Stielt_trans_2} can be written as
\[
    \sigma^4 \beta_n^2 \tilde{s}_n^2(z) + z \tilde{s}_n(z) + 1 = o(1)
\]
completing the proof of Lemma~\ref{lem:approx_recur_semi_high}
\end{proof}

\section*{Acknowledgements}
SG was supported in part by the MOE grants R-146-000-250-133, R-146-000-312-114, A-8002014-00-00 and MOE-T2EP20121-0013. SSM was partially supported by the INSPIRE research grant DST/INSPIRE/04/2018/002193 from the Dept.~of Science and Technology, Govt.~of India, and a Start-Up Grant from Indian Statistical Institute.

\bibliographystyle{abbrvnat}
\bibliography{bib}
\appendix

\clearpage
\section{Auxiliary results}\label{sec:aux}
Here we collect lemmas and results borrowed from the literature. First we define some notations.
\[
    \mathrm{Mat}_n(\mathbb C) := \text{The set of all } n \times n \text{ matrices with complex entries}.
\]
For $A \in \mathrm{Mat}_n(\mathbb C) $, define the Hilbert-Schmidt norm of $A$ by
\[
    \|A\|_\hs := \sqrt{\sum_{1 \le i \le j \le n}|A_{ij}|^2}.
\]
For $x \in \mathbb{R}^n $, let $\|x\| = \sqrt{\sum_{i = 1}^n x^2_i}$. The Operator norm of $A$ is defined as
\[
    \norm{A}_\op := \sup_{\|x\| = 1} \|Ax\|.
\]
For a matrix $A$ with eigenvalues $\lambda_1, \ldots, \lambda_n$, let $F_{A}(x) := \frac{1}{n} \sum_{i = 1}^n \mathbb{I}_{[\lambda_i, \infty)}(x)$ be the empirical distribution function associated with the eigenvalues.  Let $\mathbb S_n$ denotes the set of all permutations of the set $\{1, 2, \ldots, n\}$.
\begin{lemma}[Hoffmann-Wielandt inequality]\label{lem:hoffman-wielandt}
Let $A,B \in \mathrm{Mat}_n(\mathbb C) $ are two normal matrices, with eigenvalues $\lambda_1(A),\lambda_2(A), \ldots, \lambda_n(A)$ and $\lambda_1(B),\lambda_2(B), \ldots, \lambda_n(B)$ respectively. Then we have
\[
\min_{\sigma \in \mathbb S_n} \sum_{i = 1}^n | \lambda_i(A) - \lambda_{\sigma(i)}(B)|^2 \le \|A - B \|^2_\hs.
\]
An immediate consequence of this is that
\[
    d_{W_2}(\mu_A, \mu_B)^2 \le \frac{\|A - B\|^2}{n}.
\]
\end{lemma}
\begin{lemma}[Rank inequality]\label{lem:rank_ineq}
Let $A, B \in \mathrm{Mat}_n(\mathbb C)  $ are two Hermitian matrices. Then
\[
    \sup_{x \in \mathbb R} |F_A(x) - F_B(x)| \le \frac{\mathrm{rank}(A - B)}{n}.
\]
\end{lemma}
\begin{lemma}[Further results on perturbations of resolvents]\label{lem:matrix_perturb}
Let $C \in \mathrm{Mat}_p(\mathbb R)$ be a symmetric, positive semi-definite matrix, $y \in \mathbb{R}^p$, $z = u+\iota v \in \mathbb{C}^+$, $\varepsilon >0$, then
\begin{enumerate}
    \item [(a)] $\norm{(C - zI)^{-1}}_\op \leq 1/v$;
    \item [(b)] $|\Tr(C + y y^\top -zI)^{-1} - \Tr(C - zI)^{-1}| \leq \min \{\frac{1}{v}, \frac{ \norm{y}^2}{v^2}\}$;
    \item [(c)] $|\Tr(C + \varepsilon I -zI)^{-1} - \Tr(C -zI)^{-1}| \leq \frac{p \varepsilon}{v^2}$;
    \item [(d)] $\norm{(C + \varepsilon I -zI)^{-1} - (C -zI)^{-1}}_\op \leq \frac{\varepsilon}{v^2}$;
    \item [(e)] $|y^\top (C+y y^\top -zI)^{-1} y| \leq 1 + \frac{|z|}{v}$;
    \item [(f)] $\Im(z\Tr(C- zI)^{-1}) \geq 0$;
    \item [(g)] $\Im(\Tr(C - zI)^{-1}) >0$;
    \item [(h)] $ \Im(z y^\top (C- zI)^{-1}y) \geq 0$;
    \item [(i)] $y^\top (c+ y y^\top - zI)^{-1} y = \frac{y^\top (C-zI)^{-1} y}{1 + y^\top (C-zI)^{-1} y}$.
\end{enumerate}
\end{lemma}
In Lemma~\ref{lem:matrix_perturb}, parts (a) and (e)-(i) can be found in \cite{yaskov2016short}. Parts (b)-(d) follow from the matrix identity $A^{-1} -B^{-1} = A^{-1}(B-A)B^{-1}$ provided $A$ and $B$ are invertible. We also state the following result for quadratic form of random vectors \cite{bai2010spectral}.
\begin{lemma}[Concentration of quadratic forms of random vectors]\label{lem:matrix_quad_form_moment}
Let $A \in \mathrm{Mat}_p(\mathbb C)$ be non-random and $Y = (Y_1, Y_2, \ldots, Y_p) \in\mathbb C^p$ be a random vectors of independent entries. Suppose $\bbE Y_i = 0$, $\bbE|Y_i|^2 = 1$ and $\bbE|Y_i|^{l} \leq \nu_l$ for $l= 3, 4, \ldots, L$, for some $L \geq 4$. Then, for all $l \leq L/2$, there exists $C_l>0$ (depends only on $l$), such that
\[
    \bbE|Y^* A Y -\Tr A|^l \leq C_l((\nu_4 \Tr(AA^*))^{l/2} + \nu_{2l} (\Tr (AA^*))^{l/2}).
\]
\end{lemma}
\begin{lemma}[Burkholder's inequality]\label{lem:Burkholder}
Let $X_k$ be a complex martingale difference sequence with respect to some filtration $\{\mathcal{F}_k\}_{k \ge 1}$. Then, for $l \geq 1$,
\[
    \bbE|\sum_{k=1}^n X_k|^l \leq K_p \bbE \bigg(\sum_{k=1}^n |X_k|^2 \bigg)^{l/2},
\]
whenever $\sup_k \bbE |X_k|^l < \infty$.
\end{lemma}
\end{document}